\documentclass[10pt,reqno]{amsart}

\usepackage{amsmath}
\usepackage{amssymb}
\usepackage{amsthm}
\usepackage{mathtools}
\usepackage{yhmath}

\usepackage{geometry}
\usepackage[graphicx]{realboxes}
\usepackage{float}
\usepackage{subfigure}
\usepackage{booktabs}
\usepackage{enumitem}
\usepackage{multirow}
\usepackage{bigstrut}
\usepackage{stfloats}
\usepackage{longtable}

\usepackage[dvipsnames]{xcolor}
\usepackage{listings}
\usepackage{pifont}
\usepackage{fancyhdr}

\usepackage{tikz}
\usetikzlibrary{positioning, shapes.geometric, arrows}

\makeatletter
\renewcommand\paragraph{\@startsection{paragraph}{4}{\z@}%
	{3.25ex \@plus1ex \@minus.2ex}%
	{1.5ex \@plus.2ex}%
	{\normalfont\normalsize\bfseries}}
\makeatother

\numberwithin{equation}{section}

\theoremstyle{plain} 
\newtheorem{Theorem}{Theorem}[section]
\newtheorem{Lemma}{Lemma}[section]
\newtheorem{Corollary}{Corollary}[section]
\newtheorem{Proposition}{Proposition}[section]

\theoremstyle{definition} 
\newtheorem{Definition}{Definition}[section]

\theoremstyle{remark} 
\newtheorem{remark}{Remark}[section]

\usepackage{hyperref}
\hypersetup{
	colorlinks=true,
	linkcolor=black,
	citecolor=black,
	filecolor=black,
	urlcolor=blue
}

\usepackage{setspace}

\begin{document}
	\setstretch{0.93} 
	\setlength{\jot}{2pt}

	\setlength{\abovedisplayskip}{6pt}
	\setlength{\belowdisplayskip}{6pt}
	\setlength{\abovedisplayshortskip}{3pt}
	\setlength{\belowdisplayshortskip}{3pt}

	\setlength{\parskip}{0pt}

	\title[Global existence and blowup of relativistic dust]{On classification of dynamics for dust fluid under spherical symmetry in Schwarzschild spacetime}

	\author{Yifan Liu}
	\address{Department of Mathematics, Zhejiang Sci-Tech University, Hangzhou, 310018, China}
	\email{tliuyifan@163.com}

	\author{Shuang Miao}
	\address{School of Mathematics and Statistics, Wuhan University, Wuhan, Hubei 430072, China}
	\email{shuang.m@whu.edu.cn}

	\author{Changhua Wei}
	\address{Department of Mathematics, Zhejiang Sci-Tech University, Hangzhou, 310028, China}
	\email{chwei@zstu.edu.cn}

	\maketitle

	\begin{abstract}
		\vspace{-0.5cm}
		We investigate the dynamics of classical solutions to the dust fluid model under spherical symmetry in Schwarzschild spacetime. According to whether the solution will persist globally or develop a finite-time singularity, a precise classification of initial data is provided. Moreover, a detailed analysis on the exact blowup profile near the singularity is given.
	\end{abstract}



\section{Introduction}
\noindent The global-in-time analysis of the relativistic Euler equations on curved backgrounds constitutes a major frontier in the theory of nonlinear hyperbolic conservation laws. When the fluid pressure vanishes identically, the resulting relativistic dust model reduces to a decoupled transport-driven system where the fluid variables are propagated strictly along the timelike geodesics of the underlying geometry. In the strong-field regime of a Schwarzschild spacetime, the presence of coordinate and curvature singularities, alongside the non-trivial Christoffel symbols acting as geometric source terms, severely complicates the classical PDE framework. Because the background geometry induces strong gravitational focusing without any countervailing pressure gradients, smooth initial data generic to the system naturally degenerate in finite time, precipitating the formation of caustics and shell-crossing singularities where the mass density blows up. We investigate the global dynamics of the Schwarzschild dust fluid, providing a rigorous analysis of the competition between the dispersive decay driven by the geometry and the finite-time singularity formation dictated by the nonlinear transport structure.

On a given four-dimensional Lorentzian manifold $(\mathcal{M}, g)$, the evolution of a relativistic perfect fluid is governed by the local energy-momentum conservation law:
\begin{equation}\label{2.1}
	\nabla_{\mu}T^{\mu}_{\nu} = 0,
\end{equation}
where $\nabla_{\mu}$ denotes the covariant derivative associated with the spacetime metric $g_{\mu\nu}$. The energy-momentum tensor $T^{\mu}_\nu$ is defined by
\begin{equation}\label{tensor}
	T^{\mu}_{\nu} := \left(\rho + p\right)u^{\mu}u_{\nu} + p \delta^{\mu}_{\nu},
\end{equation}
where $\rho$ represents the proper energy density, $p$ is the fluid pressure, and $u^{\mu}$ denotes the 4-velocity field normalized such that $g_{\mu\nu}u^{\mu}u^{\nu}=-1$. For the relativistic dust model considered herein, the pressure vanishes identically ($p \equiv 0$). Throughout this paper, we employ the Einstein summation convention, where repeated indices are summed over their respective ranges $\{0, 1, 2, 3\}$. Indices are raised and lowered using the exterior Schwarzschild metric:
\begin{equation}\label{metric}
	g = -f(r) dt^2 + \frac{1}{f(r)}dr^2 + r^2(d\theta^2 + \sin^2\theta d\varphi^2),
\end{equation}
where the metric coefficient $f(r)$ is given by
\begin{equation}\label{f(r)}
	f(r) = 1 - \frac{2m}{r}.
\end{equation}
Here, $m > 0$ represents the mass of the black hole, and we restrict our domain to the exterior region $r > 2m$. The study of relativistic fluids in this setting presents a fundamental challenge in mathematical general relativity, balancing the analytical intricacies of nonlinear hyperbolic conservation laws with the geometric constraints of a fixed black hole background. We restrict our attention to spherically symmetric profiles. Specifically, we assume the 4-velocity takes the form $u = (u^{0}(t,r), u^{1}(t,r), 0, 0)$. The physical classical velocity $v = v(t,r)$ is defined relative to the coordinate frame as:
 \begin{equation}\label{velocity-def}
 	v(t,r) := cf^{-1}(r)\frac{u^1}{u^0}.
 \end{equation}
 Under this ansatz, the conservation law \eqref{2.1} for the relativistic dust fluid can be recast as the following quasilinear hyperbolic system (see Section \ref{sec:2} or \cite{Liu-Qu-Wang-Wei} for a detailed derivation):
 \begin{equation}\label{main-eq1}
 	\begin{cases}
 		\partial_t v + \frac{f(r)}{c}v\partial_r v - \frac{f'(r)}{2c}(v^2 - c^2) = 0,\\
 		\partial_t\rho + \frac{f(r)}{c}v\partial_r\rho + \rho\left(\frac{f(r)}{c}\partial_r v + \frac{2f(r)v}{cr}\right) = 0.
 	\end{cases}
 \end{equation}
 We supplement \eqref{main-eq1} with smooth initial data at $t=0$:
 \begin{equation}\label{initial-data}
 	v(0,r) = v_0(r), \quad \rho(0,r) = \rho_0(r),
 	\end{equation}
 	where $v_0(r) \in C^2([\alpha_{\min}, \alpha_{\max}])$, $\rho_0(r) \in C^1([\alpha_{\min}, \alpha_{\max}])$ is strictly positive, and the spatial boundaries satisfy:
 	\begin{equation*}
 		2m < \alpha_{\min} \leq r \leq \alpha_{\max} \leq +\infty.
 		\end{equation*}
 		To ensure the physical validity of the matter model, the initial velocity is strictly bounded by the speed of light, satisfying:
 		\begin{equation}\label{con}
 			-c < v_{\min} \leq v_0(r) \leq v_{\max} < c.
 			\end{equation}
 			\begin{remark}\label{rem:weak-hyp}
 				The characteristic eigenvalues of the quasilinear system \eqref{main-eq1} are coincident, given by:
 				\begin{equation*}
 					\lambda_1 = \lambda_2 = \frac{f(r)}{c}v,
 					\end{equation*}
 					with the corresponding linearly independent right eigenvector space collapsing to a single geometric vector $r_1 = r_2 = (0, 1)^{T}$. Consequently, system \eqref{main-eq1} is strictly weakly hyperbolic and linearly degenerate in the sense of Lax. A key analytical consequence of this algebraic structure is that the density variable $\rho$ experiences a loss of one derivative relative to the velocity field $v$. To compensate for this derivative loss and close the energy estimates, we naturally prescribe an initial regularity for the velocity field that is one order higher than that of the density.
 					\end{remark}
\subsection{Research history}
The long-time behavior of solutions to multi-dimensional relativistic Euler equations has been studied extensively in past decades. In flat spacetime a breakthrough was achieved by Christodoulou (see \cite{Christodoulou}) in which he revealed the mechanism of shock formation and described the geometry of the singularity for $3d$ relativistic Euler in Minkowski spacetime.
When the background spacetime admits a metric taking the form $$g=-dt^2+a^2(t)\sum_{i=1}^{3}(dx^i)^2,$$
we get a spatially homogeneous and isotropic cosmological spacetime, which is a class of exact solutions to Einstein field equations called Friedmann-Lema\^{\i}tre-Robertson-Walker (FLRW). For dynamical stability of the FLRW solution to the Einstein Euler equations, see \cite{Todd3,Fajman, F-T,F-T1,FOOW24,FOOW25,Friedrich,H-J,Ju-Ke-Wei,LeFloch-Wei,Liu-Wei,Todd1,Ringstrom,Ringstrom1,Rodnianski,Speck,Speck1,Todd2,Oliynyk1,Wei}.
All the above results are for the fluids in an expanding spacetime, for which one expects global-in-time stability. For spacetime admitting black holes, the situation is more diverse. In particular the dust model in this kind of spacetime is related to gravitational collapse and cosmic censorship (see \cite{Chr1,OS,Tolman}). However, the study on long-time dynamics for relativistic Euler equation in black hole spacetime is relatively limited. LeFloch and Xiang \cite{LeFloch-Xiang,LeFloch-Xiang1} derived the relativistic Euler equations evolving in Schwarzschild spacetime and studied the solution in bounded variation space. Hu and Guo \cite{Hu} investigated the spherically symmetric solution of the Chaplygin gas in Schwarzschild spacetime and identified a class of initial data that leads to a finite time singularity of the smooth solution and at the blowup time, the mass-energy density itself approaches infinity. The naked singularity and concentration of the fluids has been systematically studied recently by Guo, Had\v{z}i'{c} and Jang \cite{Guo-Hadzic-Jang} for Einstein-Euler system with $p=\epsilon \rho$ when $0<\epsilon\ll1$.  Liu, Qu, Wang and Wei \cite{Liu-Qu-Wang-Wei} studied the relativistic dust evolving in Schwarzschild-de Sitter and Schwarzschild-Anti-de Sitter spacetimes with vanishing mass $m$. An autonomous dynamical system along the characteristics, and the exact solution formulas are obtained. Based on these, an exact classification on initial data that elucidating whether the classical solution will persist globally or encounter a finite-time blowup including the coordinate singularity and fluid singularities is given. For Schwarzschild spacetime, due to the presence of a non-zero mass $m>0$, the background metric explicitly breaks this autonomy, producing a highly coupled, singular, non-autonomous ODE system along the characteristic lines. Therefore one needs more sophisticated techniques, such as an implicit function framework ($\psi(v, \alpha) = 0$) and mapping the temporal integrals into alternative phase-space tracking variables ($t \mapsto v$) to understand global dynamics for dust model and this is the main goal of the present work.
\subsection{Main results}
Assume $r>0$ and $f(r)>0$.  We investigate the fluid evolution in the exterior domain of the Schwarzschild spacetime which corresponds $r\in S:=(2m,\infty)$.
\begin{Definition}
	The lifespan $T_{*}$ is the supremum of all times $T>0$ such that the Cauchy problem \eqref{main-eq1} admits a classical solution on $[0,T]\times S$.
\end{Definition}
\subsubsection{Main results of relativistic dust in the case $v_0(r)=0$}
In this case we denote
\begin{equation*}
	R_1=\left\{2m<r\leq\frac{\sqrt{33}+3}{3}m \right\},\ R_2=\left\{r>\frac{\sqrt{33}+3}{3}m\right\},\
\end{equation*}
\begin{equation*}
	H_1=\left\{v'_0(r)\geq\chi_1(r)\right\},\ H_2=\left\{v'_0(r)<\chi_1(r)\right\},
\end{equation*}
\begin{equation*}
	H_3=\left\{v'_0(r)>\chi_2(r)\right\},\ H_4=\left\{v'_0(r)\leq\chi_2(r)\right\},
\end{equation*}
where
$$\chi_1(r)=-\frac{mc(3r-2m)}{2r^2(r-2m)}-\frac{3\sqrt{2}c}{2r}\sqrt{\frac{m}{r-2m}}\arctan\sqrt{\frac{r-2m}{2m}},$$
$$\chi_2(r)=-\frac{c}{2r}\sqrt{\frac{m}{r-2m}}\left(\sqrt{\frac{\sqrt{33}-3}{2}}+\frac{(\sqrt{33}+3)^{\frac{3}{2}}}{8\sqrt{3}}+3\sqrt{2}\arctan\sqrt{\frac{\sqrt{33}-3}{6}}\right).$$
\begin{Theorem}\label{thm:1.1}
If $v_0(r)=0$, then the Cauchy problem \eqref{main-eq1} admits a unique classical solution on $[0, +\infty)\times \left(2m,+\infty\right) $, provided that all the initial profiles satisfy
	$$(r, v_{0}^{\prime}(r))\in\mathcal{F}=(R_1\cap H_1) \cup(R_2\cap H_3).$$
	If there exists initial profiles satisfying
	$$(r, v_{0}^{\prime}(r))\in\mathcal{G}=(R_1\cap H_2) \cup(R_2\cap H_4)$$
	then there exist a $0<T_*<+\infty$ such that
	$$\|\partial_{r}v(t)\|_{L^{\infty}(S)}=O\left(\frac{1}{T_*-t}\right)\rightarrow -\infty\ \ \text{and}  \ \ \|\rho(t)\|_{L^{\infty}(S)}=O\left(\frac{1}{T_*-t}\right)\rightarrow +\infty
	$$
	as $t\rightarrow T_*^{-}$.
\end{Theorem}
\begin{remark}
	The main results of Theorem \ref{thm:1.1} can be summarized by Table \ref{eq}.
	\begin{table}[H]
		\centering
		\caption{The global existence and blowup for the case when $v_0(r)=0$}
		\label{eq}
		\vspace{5pt}
		\begin{tabular}{l|l|l}
			\hline
			\multicolumn{1}{c|}{$r$} & \multicolumn{1}{c|}{$v_0'(r)$} & \multicolumn{1}{c}{behavior of the solution} \\
			\hline
			\multirow{2}{*}{$\left(2m,\frac{\sqrt{33}+3}{3}m\right]$}& $\geq \chi_1(r)$& a unique classical solution on $[0, +\infty)\times \left(2m,+\infty\right) $ \\
			\cline{2-3}
			& $< \chi_1(r)$&$\rho$ and $\partial_{r}v$ blow up at $T_{*}$ \\
			\hline
			\multirow{2}{*}{$\left(\frac{\sqrt{33}+3}{3}m,+\infty\right)$}& $> \chi_2(r)$& a unique classical solution on $[0, +\infty)\times \left(2m,+\infty\right) $ \\
			\cline{2-3}
			& $\leq \chi_2(r)$&$\rho$ and $\partial_{r}v$ blow up at $T_{*}$ \\
			\hline
		\end{tabular}
	\end{table}
\end{remark}
\begin{remark}
	We illustrate the blowup phenomena of the relativistic dust evolving in the case when $v_0(r)=0$ by classifying the
	initial data and plotting the corresponding behaviors in Figure 1.
	According to Theorem \ref{thm:1.1}, if the initial data lie in $(R_1\cap H_1) \cup(R_2\cap H_3)$, which is the green domain, the classical solution exists globally.
	If the initial data lie in $(R_1\cap H_2) \cup(R_2\cap H_4)$, which is the red domain, the classical solution blows up in a finite time $T_*$, and as $t \to T_*$, we have
	$$v_r\rightarrow -\infty\ \ \text{and} \ \ \rho\rightarrow +\infty.
	$$
    \begin{figure}
	\begin{center}
		\includegraphics[width=0.6\linewidth]{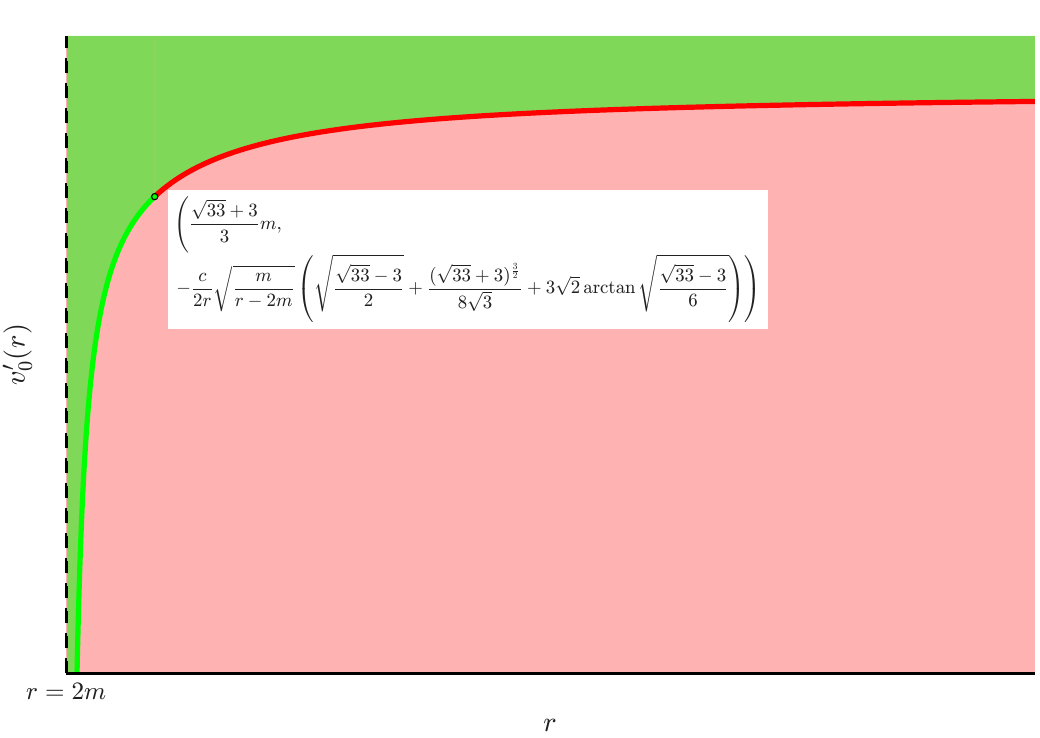}
	\end{center}
    \caption{Global existence and blowup domain in the case when $v_0(r)=0$}
    \end{figure}
\end{remark}
\subsubsection{Main results of relativistic dust in the case $v_0(r)\neq0$}
In this case, before presenting the main results, we need to introduce some notations. In the interval $r>2m>0$, we denote
\begin{equation*}
	\omega(t,r)=\frac{v^2-c^2}{1-\frac{2m}{r}},\quad D(r)=\omega_0(r)+c^2,\quad \omega_0(r):=\omega(0,r)=\frac{v^2_0(r)-c^2}{1-\frac{2m}{r}}.
\end{equation*}
We have
\begin{equation*}
	\omega_0'(r)=\frac{2r(r-2m) v_0(r)v'_0(r)-2m(v_0^2(r)-c^2)}{(r-2m)^2},
\end{equation*}
and
\begin{equation*}
	\omega_0''(r)=\frac{4v_0(r)v'_0(r)+2r\left[\left(v'_0(r)\right)^2+v_0(r)v''_0(r)\right]}{r-2m}-\frac{4r v_0(r)v'_0(r)}{(r-2m)^2}+\frac{4m(v_0^2(r)-c^2)}{(r-2m)^3}.
\end{equation*}
When $v_0(r)\neq0$, we define the following set of the initial data in $r>2m$, which leads to different dynamics of the solution.
\begin{equation*}
	V_+=\{v_0(r)>0\},\ V_0=\{v_0(r)=0\},\ V_-=\{v_0(r)<0\},
\end{equation*}
\begin{equation*}
	\Omega_+=\{\omega_0'(r)>0\},\ \Omega_0=\{\omega_0'(r)=0\},\ \Omega_-=\{\omega_0'(r)<0\},
\end{equation*}
\begin{equation*}
	D_+=\{D(r)>0\},\ D_0=\{D(r)=0\},\ D_-=\{D(r)<0\},
\end{equation*}
\begin{equation*}
	\Omega_1=\left\{\omega_0'(r)\leq \chi_3(r)\right\},\ \Omega_2=\left\{\omega_0'(r)>\chi_3(r)\right\},
\end{equation*}
where
$$\chi_3(r)=2m(c^2-v_0^2(r))v_0(r)\int_{v_0(r)}^{-c}\frac{ds}{s^2(s^2-D(r))^2}.$$
Let
\begin{equation*}
	E=V_-\cap \Omega_+,
\end{equation*}
\begin{equation*}
	F_1=V_+\cap(\Omega_+\cup \Omega_0)\cap (D_+\cup D_0),\ F_2=V_-\cap(\Omega_-\cup\Omega_0),\
\end{equation*}
\begin{equation*}
	 G_1=V_+\cap\Omega_0\cap D_-,\ G_2=V_+\cap\Omega_-.
\end{equation*}
\begin{Theorem}\label{thm:1.2}
	When $v_0(r)\neq0$, the Cauchy problem \eqref{main-eq1} admits a unique classical solution on $[0, +\infty)\times \left(2m,+\infty\right) $, provided all the initial data lie in the set
	$$\mathcal{F}=(E\cap\Omega_1) \cup F_1\cup F_2 $$
    for all $r>2m$. 	If the initial data lie in the set
	$$\mathcal{G}=(E\cap\Omega_2) \cup G_1\cup G_2,$$
	then there exists a $0<T_*<+\infty$ such that
	$$\|\partial_{r}v(t)\|_{L^{\infty}(S)}=O\left(\frac{1}{T_*-t}\right)\rightarrow -\infty\ \ \text{and} \ \ \|\rho(t)\|_{L^{\infty}(S)}=O\left(\frac{1}{T_*-t}\right)\rightarrow +\infty
	$$
	as $t\rightarrow T_*^{-}$. 	If the initial data lie in the set
	$$\mathcal{H}=V_+\cap\Omega_+\cap D_-,$$
	for some $r^{*}\in(2m,\infty)$, there exists a time $0<t^*<+\infty$ such that $v(t^*,r^{*})=0$ and the classical solution can be extended to $t^*$. The occurrence of blowup is determined by the behavior of the solution at $t=t^*$, in accordance with the blowup criterion established in Theorem \ref{thm:1.1}.
\end{Theorem}
\begin{remark} The blowup time $T_*$ stated through the whole paper may differ from each other.
	The main result of Theorem \ref{thm:1.2} can be summarized in Table \ref{neq}.
\begin{table}[H]
	\centering
	\caption{The global existence and blowup for $v_0(r)\neq0$}
	\label{neq}
	\vspace{5pt}
	\begin{tabular}{l|l|l|l}
		\hline
		\multicolumn{1}{c|}{$v_0(r)$} & \multicolumn{1}{c|}{$\omega_0'(r)$} & \multicolumn{1}{c|}{$D(r)$} & \multicolumn{1}{c}{behavior of the solution} \\
		\hline
		\multirow{5}{*}{$>0$} & \multirow{2}{*}{$>0$} & $\geq0$ & a unique classical solution on $[0, +\infty)\times \left(2m,+\infty\right) $ \\
		\cline{3-4}
		&                     & $<0$      & depend on the solution $v(t^*,r)$ and Theorem \ref{thm:1.1}\\
		\cline{2-4}
		& \multirow{2}{*}{=0} & $\geq0$ & a unique classical solution on $[0, +\infty)\times \left(2m,+\infty\right) $ \\
		\cline{3-4}
		&                     & $<0$      & $\rho$ and $v_{r}$ blow up at $T_{*}$ \\
		\cline{2-4}
		& $<0$                  & \multicolumn{2}{c}{$\rho$ and $v_{r}$ blow up at $T_{*}$} \\
		\hline
		\multirow{2}{*}{$<0$} & \multicolumn{2}{c|}{$\omega_0'(r)>\chi_3(r)$}      & $\rho$ and $v_{r}$ blow up at $T_{*}$ \\
		\cline{2-4}
		& \multicolumn{2}{c|}{$\omega_0'(r)\leq \chi_3(r)$} & a unique classical solution on $[0, +\infty)\times \left(2m,+\infty\right) $ \\
		\hline
	\end{tabular}
\end{table}
\end{remark}
\begin{remark}
	To explicitly illustrate how the solution behavior primarily depends on the initial data, we can plot them in terms of $v_0(r)$ and $\omega'_0(r)$ for a fixed $r\in [\alpha_{min}, \alpha_{max}]$, see Figure 2. According to Theorem \ref{thm:1.2}, if the initial data lie in $(E\cap\Omega_1) \cup F_1\cup F_2$, which is the green domain, the classical solution exists globally. If the initial data lie in $(E \cap \Omega_2) \cup G_1\cup G_2$, which is the red domain, the classical solution blows up in a finite time $T_*$, and as $t \to T_*$, we have
	$$v_r\rightarrow -\infty\ \ \text{and} \ \rho\rightarrow +\infty.
	$$
	If the initial data lie in $V_+\cap\Omega_+\cap D_-$, which is the yellow domain, the classical solution can be extended to the time $0<t^*<+\infty$ such that $v(t^*,r^*)=0$. In this case, the blowup criterion of Theorem \ref{thm:1.1} can be applied to $v(t^*,r)$ to determine whether the solution blows up or exists globally after $t^*$.
    \begin{figure}
    \begin{center}
    	\includegraphics[width=0.6\linewidth]{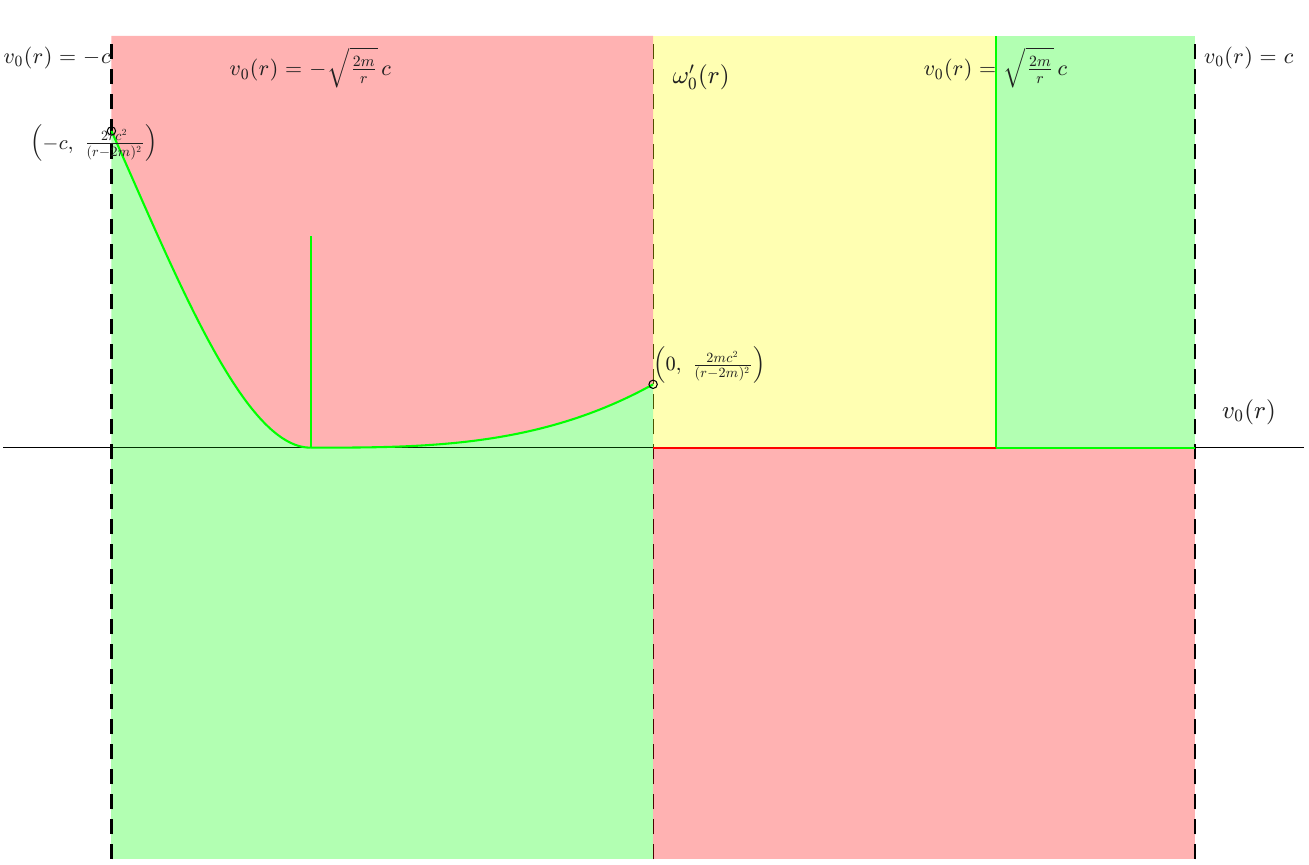}
    \end{center}
    \caption{Global existence and blowup domain for $v_0(r)\neq0$}
    \end{figure}
\end{remark}
\begin{remark}\label{rmk: type I}
As we have seen from Theorem \ref{thm:1.1} and Theorem \ref{thm:1.2}, in Schwarzschild spacetime the dust fluid only admits \emph{Type I} blowup (see its definition in \cite{Liu-Qu-Wang-Wei}), namely $\partial_{r}v$ and $\rho$ blow up. This is in contrast to Schwarzschild-anti de Sitter spacetimes where dust fluids admit \emph{type II} blowup, namely $\partial_{r}^{2}v$ and $\partial_{r}\rho$ blow up (also see \cite{Liu-Qu-Wang-Wei}).
\end{remark}

\subsection{Strategy of the proof}
We prove the main results using the method of characteristics via following steps.
\begin{itemize}
\item Along the characteristic curves defined by \eqref{r'}, we obtain the non-autonomous ODEs \eqref{v'} and \eqref{curve}. These two equations yield a fundamental conserved quantity $\omega$ as defined in \eqref{omega1}, which establishes an explicit geometric link between the velocity $v$, the characteristic curve $r$ and the metric component $f(r)$. This conserved quantity implies that the coordinate blowup ($f(r)\rightarrow 0$) is equivalent to the fluid singularity $|v|\rightarrow c$. This important observation indicates that the fluid speed approaches the speed of light as it reaches the boundary of the event horizon at $r=2m$ in Schwarzschild spacetime.

\item Evaluating the first-order spatial derivative $v_r$ along the characteristics $r=r(t,\alpha)$ shows that its finite-time explosion is driven by either the vanishing of the velocity field ($v\rightarrow 0$) or the convergence of characteristic lines ($r_{\alpha}\rightarrow 0$), see the detailed analysis in Subsection \ref{subsec2}. To study the trajectory of the velocity $v$, we derive an implicit function \eqref{psi}, based on which we obtain the exact dynamical trajectory of the velocity for different initial data settings in Subsection \ref{subsec3.1}.

\item For $r_\alpha$, we find that it satisfies a singular ODE \eqref{ra} when $v\rightarrow0$, which leads to a singular integral \eqref{ra1}. We resolve this obstruction by iteratively testing the regularizing properties of the system and performing a variable transformation of the integral variable from $t$ to $v$.

\item We discuss the dynamical behavior of $r_{\alpha}$ in Subsection \ref{subsubsec:3.2.3}, considering eight distinct regimes of the initial data.

\item  When $v_0(r)=0$,we demonstrate local well-posedness for the singular system governing $r_\alpha$, see Case 4 in Subsection \ref{subsubsec:3.2.3}. Furthermore, we find a critical curve depending not only on the monotonicity of the initial data, but also on the location of the initial fluids as well as the mass $m$ of the black hole.

\item We prove the main theorems by the classical characteristic method and provide a classification of the initial data leading to different dynamics according to the behavior of $r_{\alpha}$ discussed in the third item.
\end{itemize}
\begin{remark}
  Compared with \cite{Guo-Hadzic-Jang}, we find that for relativistic dust in a fixed Schwarzschild spacetime, a singular and non-autonomous dynamical system constitutes a main obstruction, whereas this singular system does not appear in Schwarzschild-de Sitter or Schwarzschild-Anti-de Sitter spacetimes (see \cite{Liu-Qu-Wang-Wei}). Furthermore, due to the linear degeneracy of relativistic dust model, the blowup phenomena here differs from both implosion and classical shocks.
\end{remark}

\subsection{Arrangement of the paper}
We arrange the paper as follows. In Section \ref{sec:2}, we present some necessary lemmas. Section \ref{sec:3} focuses on the detailed analysis of the solution formulas and their dynamical behaviors along the characteristic line. In the last section, we prove the two main results.

\section{Preliminaries}\label{sec:2}
\noindent In this section, we will provide some necessary preparations for the subsequent analysis. This includes the derivation of the relativistic dust in Schwarzschild spacetimes, as well as the introduction of several important lemmas. These elements are crucial for the proof of the results that will be presented later.
\subsection{Derivation of the relativistic dust in Schwarzschild space-time}
By imposing the normalization constraint on the velocity field, we obtain the following expression in the fixed background metric given by \eqref{metric}
\begin{equation*}
	-1=-f(r)(u^0)^2+f^{-1}(r)(u^1)^2.
\end{equation*}
Define the classical velocity $v=v(t,r)$ as
\begin{equation*}
	v(t,r):= cf^{-1}(r)\frac{u^1}{u^0}.
\end{equation*}
Direct calculations yield
\begin{equation*}
	(u^0)^2=f^{-1}(r)\frac{c^2}{c^2-v^2},\;
	(u^1)^2=f(r)\frac{v^2}{c^2-v^2}.
\end{equation*}

Then we can get the non-zero components of the energy momentum tensor $ T^{\mu\nu} $ as
\begin{align*}
	&T^{00}=f^{-1}(r)\frac{\rho c^4+pv^2}{c^2-v^2},
	\quad T^{01}=T^{10}=\frac{(\rho c^2+p)cv}{c^2-v^2},\\
	&T^{11}=f(r)\frac{(\rho c^2+p)c^2}{c^2-v^2},\quad
	T^{22}=\frac{p}{r^2},\quad
	T^{33}=\frac{p}{r^2\sin^2\theta}.
\end{align*}

Expanding the relativistic Euler equation \eqref{2.1} directly, we have
\begin{equation*}
	\frac{\partial T^{\eta\nu} }{\partial x^\eta}+\Gamma_{\eta\mu}^\eta T^{\mu\nu}+\Gamma_{\eta\mu}^\nu T^{\eta\mu}=0.
\end{equation*}

By letting the pressure $p = 0$, we can obtain the relativistic dust equations
\begin{equation}\label{ME1}
	\begin{cases}
		\frac{\partial}{\partial t}\left(\frac{f^{-1}(r)\rho c^4}{c^2-v^2}\right) + \frac{\partial}{\partial r}\left(\frac{\rho c^3v}{c^2-v^2}\right) + \left(\frac{2}{r} + \frac{f'(r)}{f(r)}\right)\frac{\rho c^3v}{c^2-v^2} = 0,\\
		\frac{\partial}{\partial t}\left(\frac{\rho c^3v}{c^2-v^2}\right) + \frac{\partial}{\partial r}\left(\frac{f(r)\rho c^2v^2}{c^2-v^2}\right) + \frac{1}{2}f'(r)\frac{\rho c^4}{c^2-v^2}+ \left[\frac{2f(r)}{r}-\frac{1}{2}f'(r)\right]\frac{\rho c^2v^2}{c^2-v^2}\ = 0.
	\end{cases}
\end{equation}
In the existence domain of the classical solution, system \eqref{ME1} is equivalent to
\begin{equation}\label{eq1}
	\begin{cases}
		\partial_t v+\frac{vf(r)}{c}\partial_r v-\frac{f'(r)}{2c}(v^2-c^2)=0,\\
		\partial_t\rho+\frac{vf(r)}{c}\partial_r\rho+\rho\left(\frac{f(r)}{c}\partial_rv+\frac{2f(r)v}{cr}\right)=0.
	\end{cases}
\end{equation}

\subsection{Some necessary lemmas}

The first lemma is the well-known Hadamard's lemma:
\begin{Lemma}\label{rem:2.1}
	A $C^{1}$ \textit{mapping} $x:\mathbb{R}^n\to \mathbb{R}^n,x=\varphi(\alpha)$ is a global $C^{1}$
	diffeomorphism if and only if
	\begin{itemize}
		\item $\varphi$ is a proper mapping, namely,
		$$|\varphi(\alpha)|\to\infty \quad as \quad|\alpha|\to\infty .$$
		\item $\varphi$ is a local $C^1$ diffeomorphism, i.e.,
		$$\det\left(\frac{\partial\varphi_i(\alpha)}{\partial\alpha_j}\right)\neq0,\quad\forall \alpha\in\mathbb{R}^n.$$
	\end{itemize}
\end{Lemma}
The next lemma is the well-known mean value theorems in Calculus, which is important for estimating the blowup behavior near the blowup point.
\begin{Lemma}\label{lem:2.2}
	If $f(x)$ and $g(x)$ are continuous functions defined on the interval $[a,b]$, and $g(x)$ does not change its sign on $[a,b]$, then there exists $\theta\in [a,b]$ such that
	\begin{equation}
		\int_{a}^{b} f(x)g(x) dx =f(\theta)\int_{a}^{b} g(x) dx.
	\end{equation}
\end{Lemma}
According to Lemma \ref{lem:2.2}, we have the following corollary

\begin{Corollary}\label{cor:2.1}
	If $f(x)$ is a continuous function and does not change its sign on the interval $[-\delta_1,-\delta_2]$, where $\delta_1\geq0$ and $\delta_2\geq 0$, we have
	\begin{equation*}
		\int_{-\delta_1}^{-\delta_2} \frac{f(x)}{x} dx= f(\theta)\ln \frac{\delta_2}{\delta_1},
	\end{equation*}
	where $\theta \in [-\delta_1, -\delta_2]$. Moreover, if $f(x)>0$ on $[-\delta_1,0]$, we have
	\begin{equation*}
		\lim_{\delta_2\rightarrow 0}\int_{-\delta_1}^{-\delta_2} \frac{f(x)}{x} dx\rightarrow -\infty.
	\end{equation*}
\end{Corollary}



\section{Solution formulas and properties of relativistic dust}\label{sec:3}
Throughout this paper, we employ the classical characteristic method to investigate the global existence and blowup behavior of classical solutions. By analyzing the non-autonomous dynamical system along characteristics, we derive solution formulas by the implicit functions.
\subsection{Characteristic equations and a conserved quantity}\label{subsec3.1}

When $r_0=\alpha>2m$, the characteristic curve $\gamma(t,\alpha):=(t,r(t,\alpha))$ starting from $\alpha$ writes
\begin{equation}\label{r'}
\begin{cases}
	\frac{dr(t,\alpha)}{dt}=\left(1-\frac{2m}{r}\right)\frac{v}{c},\  \ t>0, \\
	r(0,\alpha)=\alpha,\ \ \ \ \ \ \ \ \ \quad  \alpha>2m.
\end{cases}
\end{equation}
Then along the characteristic curve, we get the following ODEs satisfied by $v(t,r(t,\alpha)):=v(t,\alpha)$ and $\rho(t,r(t,\alpha)):=\rho(t,\alpha)$
\begin{equation}\label{v'}
\begin{cases}
	\frac{dv(t,\alpha)}{dt}=\frac{m}{cr^2}\left(v^2-c^2\right),\ \ \ t>0,\\
	v(0,\alpha)=v_0(\alpha),  \ \ \ \ \  \ \  \ \ \ \alpha>2m
\end{cases}
\end{equation}
and
\begin{equation}\label{curve}
\begin{cases}
	\frac{d\rho(t,\alpha)}{dt}=-\rho X, \ \ \ \ \ t>0, \\
	\rho(0,\alpha)=\rho_0(\alpha),\ \ \ \alpha>2m,
\end{cases}
\end{equation}
where
\begin{equation}\label{X}
X=\frac{f(r)}{c}v_r+\frac{2f(r)v}{cr}.
\end{equation}

If we take
\begin{equation}\label{omega1}
	\omega(t,\alpha)=\omega(t,r(t,\alpha))=\frac{v^2-c^2}{f(r)},
\end{equation}
then along the characteristic curve $\gamma$, we have
\begin{equation*}
	\frac{d \omega}{d t}=\frac{2vv_tf(r)-f'(r)r_t(v^2-c^2)}{f^2(r)}=0.
\end{equation*}
This implies
\begin{equation}\label{omega2}
	\omega(t,\alpha)\equiv \omega_0(\alpha)=\frac{v_0^2(\alpha)-c^2}{1-\frac{2m}{\alpha}}<0,
\end{equation}
since $v^2_0(\alpha)<c^2$.

From \eqref{omega1} and \eqref{omega2}, we have
\begin{equation}\label{v}
	v^2=f(r)\omega_0(\alpha)+c^2.
\end{equation}

\subsection{Solution formulas of $v$ and $\rho$}\label{subsec3.1}
In the Schwarzschild spacetime, we can not derive the explicit expressions for classical solutions from the above ODEs. Therefore, in the following section, we will employ implicit function expressions.

Integrating \eqref{curve} directly, we can arrive at
\begin{equation}\label{ans1}
    \rho(t,\alpha)=\rho_0(\alpha)e^{-\int_{0}^{t} Xdt              }.
\end{equation}
The following remark gives the main reason for the concentration of the density $\rho$.
\begin{remark}\label{rmk: blowup of rho and deri v}
    From \eqref{ans1}, we can see that the density $\rho$ itself depends on the first order derivative of $v$, so $\rho$ losses one degree of differentiability than $v$.
    And from \eqref{X} and \eqref{ans1}, we can deduce that if $X$ remains bounded on $[0, t]$, for any finite time $t$, then the density $\rho$ remains bounded. If $X\rightarrow -\infty$ as $t \to T_*^-$ for some finite $T_*$, i.e., $X$ diverges in finite time, then $\rho\rightarrow +\infty$. Furthermore, under the conditions $r > 2m > 0$ and $v$  bounded, the divergence $\rho \to +\infty$ is due to $v_r \to -\infty$.
\end{remark}
From \eqref{v}, we can obtain
\begin{equation}\label{r}
	r=\frac{2m\omega_0(\alpha)}{\omega_0(\alpha)+c^2-v^2},
\end{equation}

\begin{equation}
	v=\pm\sqrt{\left(1-\frac{2m}{r}\right)\omega_0(\alpha)+c^2},
\end{equation}
and
\begin{equation}\label{vr}
	v_r=\frac{v_\alpha}{r_\alpha}=\frac{m\omega_0(\alpha)}{r^2v}+\frac{(1-\frac{2m}{r})\omega'_0(\alpha)}{2vr_\alpha},
\end{equation}
where
\begin{equation}\label{w0'}
	\omega_0'(\alpha)=\frac{2\alpha(\alpha-2m) v_0(\alpha)v'_0(\alpha)-2m(v_0^2(\alpha)-c^2)}{(\alpha-2m)^2}.
\end{equation}
So we further have
\begin{equation}\label{X1}
	X=\frac{f(r)}{c}v_r+\frac{2f(r)v}{cr}=\frac{1-\frac{2m}{r}}{c}\left(\frac{m\omega_0(\alpha)}{r^2v}+\frac{\left(1-\frac{2m}{r}\right)\omega_0'(\alpha)}{2vr_\alpha}+\frac{2v}{r}\right).
\end{equation}

Combining \eqref{v'} and \eqref{r}, we can get an autonomous ODE
\begin{equation}\label{v'1}
	\begin{cases}
		\frac{dv(t,\alpha)}{dt}=\frac{(v^2-c^2)(D(\alpha)-v^2)^2}{4mc\omega_0^2(\alpha)},\ \ t>0,\\
		v(0,\alpha)=v_0(\alpha),\qquad\ \ \ \ \ \ \ \alpha>2m,
	\end{cases}
\end{equation}
where $D(\alpha)=\omega_0(\alpha)+c^2=\frac{v_0^2(\alpha)-\frac{2m}{\alpha}c^2}{1-\frac{2m}{\alpha}}$.

Note that
$$D(\alpha)-v^2=\omega_0(\alpha)+c^2-v^2=\omega+c^2-v^2=\frac{2m(v^2-c^2)}{r-2m},$$
and $v^2<c^2$ from \eqref{v}, we conclude that $v(t,\alpha)$ decreases for $t\in (0,\infty)$.

Solve \eqref{v'1}, we arrive at
\begin{equation}\label{psi}
	\psi(v,\alpha)=\phi(v,\alpha)-t-k(\alpha)=0,
\end{equation}
where
\begin{equation*}
	\phi(v,\alpha)=
	\begin{cases}
		2m\ln \left(\frac{c-v}{c+v}\right)+\left(\frac{3mc}{\sqrt{D(\alpha)}}-\frac{mc^3}{D^{\frac{3}{2}}(\alpha)}\right)\ln \left|\frac{v+\sqrt{D(\alpha)}}{v-\sqrt{D(\alpha)}}\right|-\frac{2mc\omega_0(\alpha)v}{(v^2-D(\alpha))D(\alpha)}, \quad \mbox{if}\ D(\alpha)>0,\\
		\frac{2m\omega_0^2(\alpha)}{c^4}\ln \left(\frac{c-v}{c+v}\right)+\frac{4m\omega_0^2(\alpha)}{c^3v}+\frac{4m\omega_0^2(\alpha)}{3cv^3},\;\;\,\quad\quad\quad\quad\quad\quad\quad \quad\quad\quad\quad \mbox{if}\  D(\alpha)=0,\\
		2m\ln \left(\frac{c-v}{c+v}\right)+\frac{2mc(\omega_0(\alpha)+2D(\alpha))}{(-D)^{\frac{3}{2}}(\alpha)}\arctan\frac{v}{\sqrt{-D(\alpha)}}-\frac{2mc\omega_0(\alpha)v}{(v^2-D(\alpha))D(\alpha)} ,\quad\   \mbox{if}\ D(\alpha)<0,
	\end{cases}
\end{equation*}
and
\begin{equation*}
	k(\alpha)=\psi(v_0(\alpha),\alpha)=
	\begin{cases}
		2m\ln \left(\frac{c-v_0(\alpha)}{c+v_0(\alpha)}\right)+\left(\frac{3mc}{\sqrt{D(\alpha)}}-\frac{mc^3}{D^{\frac{3}{2}}(\alpha)}\right)\ln \left|\frac{v_0(\alpha)+\sqrt{D(\alpha)}}{v_0(\alpha)-\sqrt{D(\alpha)}}\right|+\frac{c\alpha v_0(\alpha)}{D(\alpha)}, \quad \mbox{if}\ D(\alpha)>0,\\
		\frac{2m\omega_0^2(\alpha)}{c^4}\ln \left(\frac{c-v_0(\alpha)}{c+v_0(\alpha)}\right)+\frac{4m\omega_0^2(\alpha)}{c^3v_0(\alpha)}+\frac{4m\omega_0^2(\alpha)}{3cv_0^3(\alpha)},\quad\quad\quad\quad\quad\quad\quad\quad\quad\quad\;\, \mbox{if}\ D(\alpha)=0,\\
		2m\ln \left(\frac{c-v_0(\alpha)}{c+v_0(\alpha)}\right)+\frac{2mc(\omega_0(\alpha)+2D(\alpha))}{(-D)^{\frac{3}{2}}(\alpha)}\arctan\frac{v_0(\alpha)}{\sqrt{-D(\alpha)}}+\frac{c\alpha v_0(\alpha)}{D(\alpha)},\quad \quad\;\,\,\, \mbox{if}\ D(\alpha)<0.
	\end{cases}
\end{equation*}

Let $v_{D}(\alpha)=\sqrt{\frac{2m}{\alpha}}c$, we have if $v_0(\alpha)=\pm v_{D}(\alpha)$, then $D(\alpha)=0$.

 The following three propositions establish the range of $v$, which is closely related to the symbol of $D(\alpha)$.

\begin{Proposition}\label{Prop1}
	Suppose that $D(\alpha)>0$, then if $v_0(\alpha)>0$, we have $\sqrt{D(\alpha)}<v\leq v_0(\alpha)<c$, and if $v_0(\alpha)<0$, we get $-c<v\leq v_0(\alpha)<-\sqrt{D(\alpha)}$ .
\end{Proposition}
\begin{proof}
	By simple calculation, we obtain
	\begin{equation*}
		v^2_0(\alpha)-D(\alpha)=v^2_0(\alpha)-\frac{v_0^2(\alpha)-c^2}{1-\frac{2m}{\alpha}}+c^2=\frac{2m(c^2-v_0^2(\alpha))}{\alpha-2m}>0.
	\end{equation*}

	Given $D(\alpha)>0$, and note that $\frac{dv}{dt}<0$ from \eqref{v'}, we have \begin{equation*}
        -c<v_0(\alpha)<-\sqrt{D(\alpha)},
    \end{equation*} or $$\sqrt{D(\alpha)}<v_0(\alpha)<c,$$ and
	\begin{equation}\label{phi}
		\phi=2m\ln \left(\frac{c-v}{c+v}\right)+\left(\frac{3mc}{\sqrt{D(\alpha)}}-\frac{mc^3}{D^{\frac{3}{2}}(\alpha)}\right)\ln \frac{v+\sqrt{D(\alpha)}}{v-\sqrt{D(\alpha)}}-\frac{2mc\omega_0(\alpha)v}{(v^2-D(\alpha))D(\alpha)}.
	\end{equation}

	By \eqref{psi} and \eqref{phi}, we can get that if $v\rightarrow\sqrt{D(\alpha)}^+$ or $v\rightarrow-c$, $$\phi\rightarrow+\infty.$$
    Combining this with \eqref{psi}, we could conclude that if $v_0(\alpha)<0$
	$$v(t,\alpha)\rightarrow-c,$$
    as $t\rightarrow +\infty$.
    If $v_0(\alpha)>0$
	$$v(t,\alpha)\rightarrow\sqrt{D(\alpha)},$$
   as $t\rightarrow +\infty$.

	So we finally obtain if $v_0(\alpha)<0$ and $t< \infty$
	$$-c<v(t,\alpha)<v_0(\alpha)<-\sqrt{D(\alpha)},$$
    and if $v_0(\alpha)>0$
	$$\sqrt{D(\alpha)}<v(t,\alpha)<v_0(\alpha)<c,$$
    for $t< \infty$.
\end{proof}

By similar discussions, we have the following two propositions
\begin{Proposition}\label{Prop2}
	Suppose that $D(\alpha)=0$, we have if $v_0(\alpha)>0$, we have
    $$0<v\leq v_0(\alpha)=v_D(\alpha),$$ and if $v_0(\alpha)<0$, we have $$-c<v\leq v_0(\alpha)=-v_D(\alpha).$$
\end{Proposition}
\begin{Proposition}\label{pro:3.3}
	Suppose that $D(\alpha)<0$, we have $$-c<v<v_0(\alpha).$$
\end{Proposition}
\begin{remark}
    The trajectory of $v$ according to the symbol of $D(\alpha)$ can be represented in Figure 3

    \begin{figure}[H]\label{fig:3}
    	\centering
    	\subfigure[The trajectory of $v$ when $D(\alpha)>0$]{
    		\includegraphics[width=0.45\linewidth]{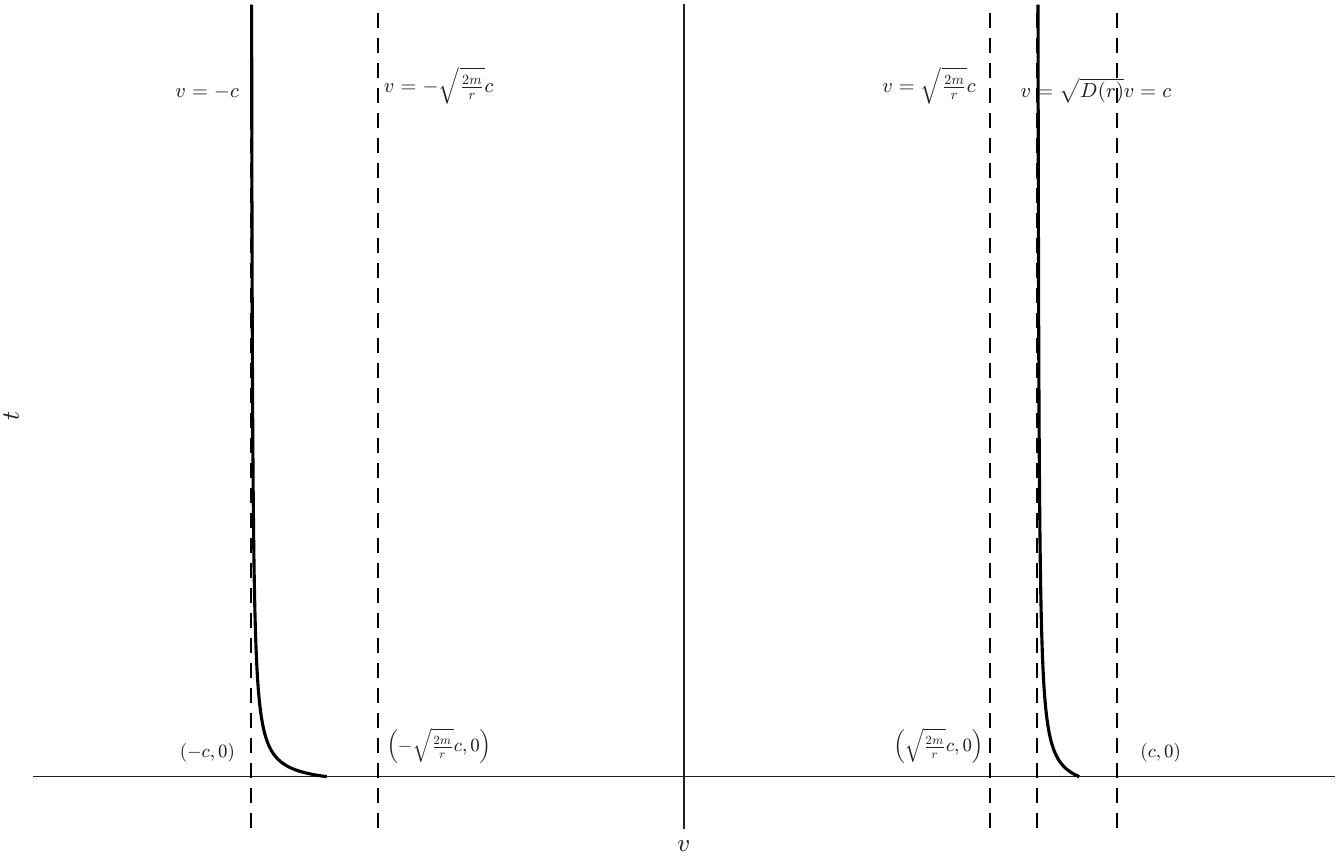}
    		\label{fig:sub3}
    	}
    	\subfigure[The trajectory of $v$ when $D(\alpha)=0$]{
    		\includegraphics[width=0.45\linewidth]{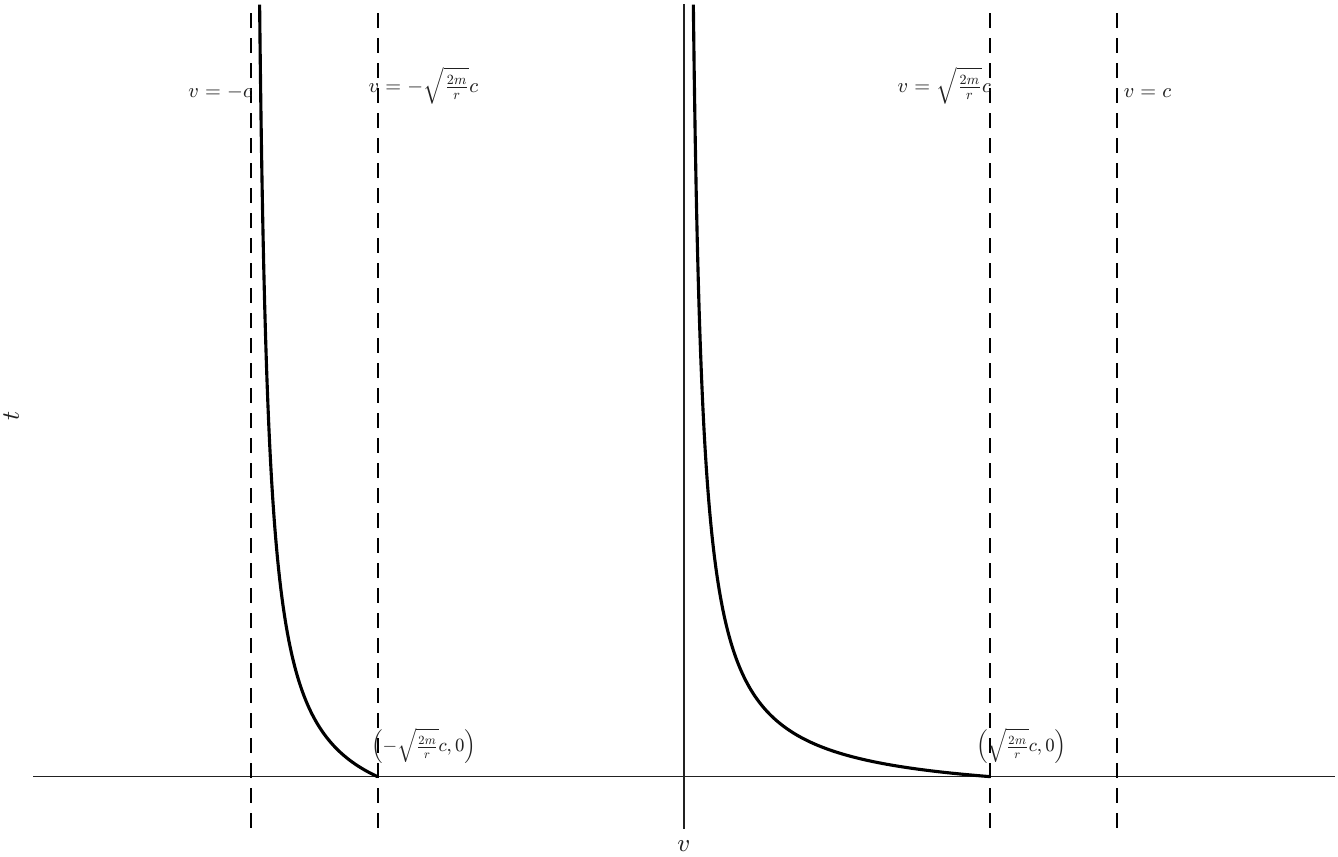}
    		\label{fig:sub4}
    	}
    	\caption{When $D(\alpha)\geq0$, the trajectory of $v$}
    	\label{fig:subfigures1}
    \end{figure}

\begin{figure}[H]
    \centering
    \subfigure[The first type of trajectory of $v$ when $D(\alpha)<0$]{
    	\includegraphics[width=0.45\linewidth]{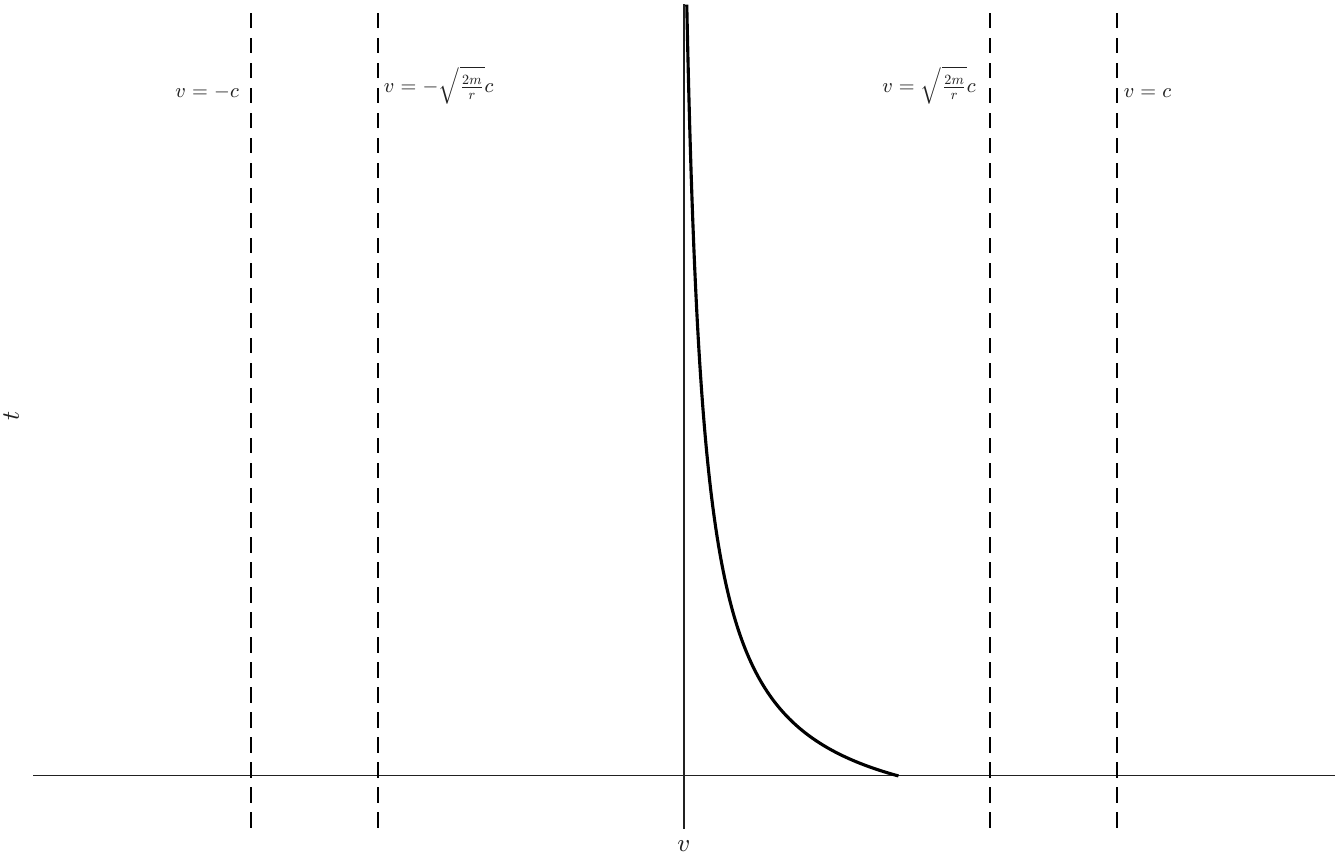}
    	\label{fig:sub5}
    }
    \subfigure[The second type of trajectory of $v$ when $D(\alpha)<0$]{
    	\includegraphics[width=0.45\linewidth]{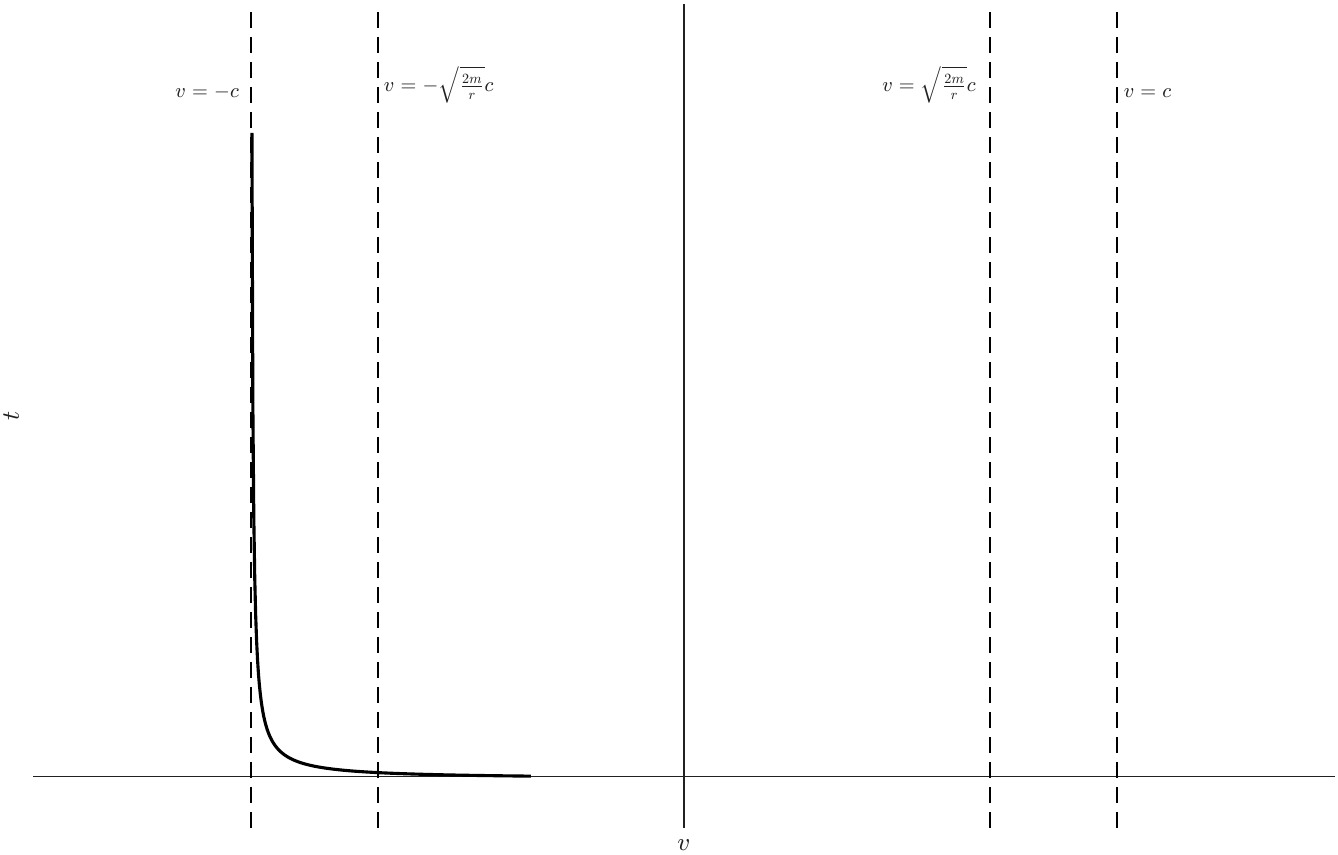}
    	\label{fig:sub6}
    }

    \caption{When $D(\alpha) < 0$, the two possible trajectories of $v$	}
    \label{fig:subfigures2}
\end{figure}
When $D(\alpha)<0$,there are two possible trajectories for
$v$: If $v_0(\alpha)>0$, the $v_{r}$ and $\rho$ may blow up as $v\to 0$ as shown in Figure \ref{fig:sub5}. In other cases, $v\to -c$ as $t\to \infty$ as shown in Figure \ref{fig:sub6}.
\end{remark}

\subsection{Mechanism and blowup behavior of $v_r$}\label{subsec2}
In this subsection, we first show that $v_{\alpha}(t,\alpha)$ is bounded along the characteristic $r(t,\alpha)$ starting from $\alpha$.
\begin{Lemma}\label{lem:3.1}
    Along the characteristic curve $r(t,\alpha)$ starting from any fixed $\alpha$, we have $v_\alpha$ is bounded for all $t$.
\end{Lemma}
\begin{proof}
    Differentiating \eqref{v'} and \eqref{v} with respect to $\alpha$ respectively, we obtain
    \begin{equation}\label{eq3}
        \frac{d v_\alpha}{d t}=\frac{2mvv_\alpha}{cr^2}-\frac{2mr_\alpha(v^2-c^2)}{cr^3},
    \end{equation}
    and
    \begin{equation}\label{eq2}
        2vv_\alpha=\left(1-\frac{2m}{r}\right)\omega_0'(\alpha)+\frac{2m\omega_0(\alpha)}{r^2}r_\alpha.
    \end{equation}
    Combining \eqref{eq3} and \eqref{eq2}, we have
    \begin{equation}\label{vat1}
    \begin{cases}
        \frac{d v_\alpha}{d t}=\frac{2v(3m-r)}{cr^2}v_\alpha+\frac{\left(1-\frac{2m}{r}\right)^2}{cr}\omega_0'(\alpha),\\
     v_\alpha(0,\alpha)=v_0'(\alpha).
    \end{cases}
    \end{equation}
Solving \eqref{vat1}, we obtain
    \begin{equation}\label{v_alpha}
        v_\alpha(t,\alpha)=e^{\int_{0}^{t}\frac{2v(3m-r)}{cr^2}ds}\left(\int_{0}^{t}\frac{\left(1-\frac{2m}{r}\right)^2 \omega_0'(\alpha)  }{cr}e^{-\int_{0}^{s}\frac{2v(3m-r)}{cr^2}d\tau} ds+v'_0(\alpha)   \right).
    \end{equation}
 By \eqref{v'1}, we may change the variable of integration from $t$ to $v$
    \begin{equation*}
        \begin{aligned}
            \int_{0}^{t}\frac{2v(3m-r)}{cr^2}ds&=\int_{v_0(\alpha)}^{v}\frac{2s(\omega_0(\alpha)+3c^2-3s^2)}{(s^2-c^2)(\omega_0(\alpha)+c^2-s^2)}ds=\int_{v^2_0(\alpha)}^{v^2}\frac{1}{u-c^2}-\frac{2}{\omega_0(\alpha)+c^2-u}du\\
            &=\ln \frac{(c^2-v^2)(\omega_0(\alpha)+c^2-v^2)^2}{(c^2-v_0^2(\alpha))(\omega_0(\alpha)+c^2-v_0^2(\alpha))^2}.
        \end{aligned}
    \end{equation*}
    Then in terms of $v$ instead of $t$ for fixed $\alpha$, we can obtain
    \begin{equation}\label{Im}
        \begin{aligned}
            &v_\alpha(t,\alpha)=e^{\int_{0}^{t}\frac{2v(3m-r)}{cr^2}ds}\left(\int_{0}^{t}\frac{\left(1-\frac{2m}{r}\right)^2 \omega_0'(\alpha)  }{cr}e^{-\int_{0}^{s}\frac{2v(3m-r)}{cr^2}d\tau} ds+v'_0(\alpha)   \right)\\
            =&\frac{(c^2-v^2)(D(\alpha)-v^2)^2}{(c^2-v_0^2(\alpha))(D(\alpha)-v_0^2(\alpha))^2}\left(2\omega_0'(\alpha)\left(1-\frac{2m}{\alpha}\right)(D(\alpha)-v_0^2(\alpha))^2\int_{v_0(\alpha)}^{v}\frac{1 }{(D(\alpha)-s^2)^3}ds+v'_0(\alpha)   \right)\\
            :=&v_\alpha(v,\alpha).
        \end{aligned}
    \end{equation}
Notice that
    \begin{equation*}
        \omega_0(\alpha)+c^2-v^2=\omega+c^2-v^2=\frac{2m(v^2-c^2)}{r-2m}<0,
    \end{equation*}
    this implies that
    $$\frac{1 }{(D(\alpha)-v^2)^3}<\infty$$ for all $|v|<c$, and $$|v_\alpha|<\infty$$ for all $t<\infty$, since $v_0(\alpha)\in C_0^1([\alpha_{min},\alpha_{max}])$.
 \end{proof}
\begin{remark}\label{vad}
Indeed,	we can obtain the expression for $v_\alpha$ \eqref{Im} directly in terms of $v$ according to different symbol of $D(\alpha)$, which is useful to analyze the behavior of $v_{rr}$ and $\rho_r$ in the following section.  When $D(\alpha)>0$, we have
	\begin{align*}
		v_\alpha=&-\frac{\omega_0'(\alpha)(c^2-v^2)v}{2D(\alpha)\omega_0(\alpha)}+\frac{\omega_0'(\alpha)v_0(\alpha)(c^2-v^2)(D(\alpha)-v^2)^2}{2D(\alpha)\omega_0(\alpha)(D(\alpha)-v_0^2(\alpha))^2}\\
		&+\frac{3\omega_0'(\alpha)(c^2-v^2)(v^2-D(\alpha))v}{4D^2(\alpha)\omega_0(\alpha)}-\frac{3\omega_0'(\alpha)v_0(\alpha)(c^2-v^2)(v^2-D(\alpha))^2}{4D^2(\alpha)\omega_0(\alpha)(v_0^2(\alpha)-D(\alpha))}\\
		&-\frac{3\omega_0'(\alpha)(c^2-v^2)(D(\alpha)-v^2)^2}{8D^{\frac{5}{2}}(\alpha)\omega_0(\alpha)}\ln\frac{(v-\sqrt{D(\alpha)})(v_0(\alpha)+\sqrt{D(\alpha)})}{(v+\sqrt{D(\alpha)})(v_0(\alpha)-\sqrt{D(\alpha)})}\\
		&+\frac{v'_0(\alpha)(c^2-v^2)(D(\alpha)-v^2)^2}{(c^2-v^2_0(\alpha))(D(\alpha)-v^2_0(\alpha))^2}.
	\end{align*}
	When $D(\alpha)=0$, we have
	\begin{align*}
		v_\alpha=\frac{v'_0(\alpha)(c^2-v^2)v^4}{(c^2-v_0^2(\alpha))v_0^4(\alpha)}+\frac{2\omega_0'(\alpha)(c^2-v^2)}{5c^2}\left(\frac{1}{v}-\frac{v^4}{v_0^5(\alpha)}\right).
	\end{align*}
	When $D(\alpha)<0$, we have
	\begin{align*}
		v_\alpha=&-\frac{\omega_0'(\alpha)(c^2-v^2)v}{2D(\alpha)\omega_0(\alpha)}+\frac{\omega_0'(\alpha)v_0(\alpha)(c^2-v^2)(D(\alpha)-v^2)^2}{2D(\alpha)\omega_0(\alpha)(D(\alpha)-v_0^2(\alpha))^2}\\
		&+\frac{3\omega_0'(\alpha)(c^2-v^2)(v^2-D(\alpha))v}{4D^2(\alpha)\omega_0(\alpha)}-\frac{3\omega_0'(\alpha)v_0(\alpha)(c^2-v^2)(v^2-D(\alpha))^2}{4D^2(\alpha)\omega_0(\alpha)(v_0^2(\alpha)-D(\alpha))}\\
		&-\frac{3\omega_0'(\alpha)(c^2-v^2)(D(\alpha)-v^2)^2}{8(-D(\alpha))^{\frac{5}{2}}(\alpha)\omega_0(\alpha)}\left(\arctan\frac{v}{\sqrt{-D(\alpha)}}-\arctan\frac{v_0(\alpha)}{\sqrt{-D(\alpha)}}\right)\\
		&+\frac{v'_0(\alpha)(c^2-v^2)(D(\alpha)-v^2)^2}{(c^2-v^2_0(\alpha))(D(\alpha)-v^2_0(\alpha))^2}.
	\end{align*}
\end{remark}

\subsubsection{Asymptotic behavior of $v$ near $v(t^*,\alpha)=0$}\label{subsection3.2.1}

The following Proposition gives the sufficient and necessary conditions on $v\rightarrow0$ in finite time.
\begin{Proposition}\label{Prop4}
	For fixed $\alpha>2m$, if $D(\alpha)<0$ and $v_0(\alpha)>0$, then $v(t,\alpha)$ vanishes at some finite time $t^*>0$. Otherwise, $v\neq 0$ for all finite time if $v_0(\alpha)\neq0$.
\end{Proposition}
\begin{proof}
	By (\ref{v}), we have
	\begin{equation*}
		v^2(t,\alpha)=c^2+\left(1-\frac{2m}{r}\right)\omega_0(\alpha)=D(\alpha)-\frac{2m}{r}\omega_0(\alpha),
	\end{equation*}
	then we have
	\begin{equation*}
		v^2>D(\alpha)>0, \quad \mbox{ if } D(\alpha)>0,
	\end{equation*}
    and
	\begin{equation*}
		v^2=-\frac{2m}{r}\omega_0(\alpha)=\frac{2mc^2}{r}>0, \quad\mbox{ if } D(\alpha)=0.
	\end{equation*}

	And when $D(\alpha)<0$, we have
	\begin{equation*}
		\phi(v)=2m\ln \left(\frac{c-v}{c+v}\right)+\frac{2mc(\omega_0(\alpha)+2D(\alpha))}{(-D(\alpha))^{3/2}}\arctan\frac{v}{\sqrt{-D(\alpha)}}-\frac{2mc\omega_0(\alpha)v}{(v^2-D(\alpha))D(\alpha)}.
	\end{equation*}
It is easy to check that $\phi(v)$ admits the following properties
	$$\frac{d\phi}{dv}=\frac{4mc\omega_0^2(\alpha)}{(v^2-c^2)(D(\alpha)-v^2)^2}<0,\quad \phi(0)=0$$
and $$\phi(v(t,\alpha))=t+k(\alpha)\rightarrow+\infty \mbox{ when } t\rightarrow\infty.
	$$
	Thus, $k(\alpha)=\phi(v_0(\alpha),\alpha)<0$ if and only if $v_0(\alpha)>0$.

    Then if $k(\alpha)< 0$ initially, there exists  $0<t^*=-\phi(v_0(\alpha),\alpha)<\infty$, such that $\phi(v(t^*,\alpha))=0$, which implies $v(t^*,\alpha)=0$.

\end{proof}
\begin{remark}\label{Taylor}
    When $v(t^*,\alpha)=0$, by \eqref{v'1}, we have
    \begin{equation*}
        \left.\frac{dv(t,\alpha)}{dt}\right|_{t=t^*}=-\frac{c(\omega_0(\alpha)+c^2)^2}{4m\omega_0^2(\alpha)}<0,
    \end{equation*}
    so we can expend $v(t,\alpha)$ at $t^*$ with Taylor expension: for any $\varepsilon>0$, there exists $\delta_1>0$ such that
    \begin{equation}\label{t0}
        \left|v(t,\alpha)-\left(-\frac{c(\omega_0(\alpha)+c^2)^2}{4m\omega_0^2(\alpha)}\right)(t-t^*)\right|<\varepsilon,\ \mbox{when } |t-t^*|<\delta_1.
    \end{equation}
\end{remark}

\subsubsection{Solution formula of the Jacobian $r_\alpha(t,\alpha)$}\label{subsection3.2.2}

Differentiate (\ref{r'}) with respect to $\alpha$, we have
\begin{equation}\label{ra}
	\frac{dr_\alpha(t,\alpha)}{dt}=\frac{3mv^2-mc^2}{cr^2v}r_\alpha+\frac{(1-\frac{2m}{r})^2\omega_0'(\alpha)}{2cv}.
\end{equation}
Integrating \eqref{ra} from $0$ to $t$, we obtain
\begin{equation}\label{ra1}
	r_\alpha(t,\alpha)=e^{\int_{0}^{t}\frac{3mv^2-mc^2}{cr^2v}ds}\left[\int_{0}^{t}\frac{(1-\frac{2m}{r})^2\omega_0'(\alpha)}{2cv}e^{ -\int_{0}^{s}\frac{3mv^2-mc^2}{cr^2v}d\tau}ds+1\right].
\end{equation}

We can rewrite \eqref{v'} as
\begin{equation*}
	\frac{1}{r^2v}=\frac{cv_t}{mv(v^2-c^2)},
\end{equation*}
thus
\begin{equation*}
	e^{\int_{0}^{t}\frac{3mv^2-mc^2}{cr^2v}ds}=e^{\int_{0}^{t}\frac{(3v^2-c^2)v'(s)}{v(v^2-c^2)}ds}=\frac{v(t)(c^2-v^2(t))}{v_0(\alpha)(c^2-v^2_0(\alpha))}.
\end{equation*}

Combining \eqref{v}, \eqref{v'1} and changing the variable of integration from $t$ to $v$, we have
\begin{eqnarray}\label{ra2}
	r_\alpha&=&\frac{v(c^2-v^2)}{v_0(\alpha)(c^2-v^2_0(\alpha))}\left[-2m(c^2-v_0^2(\alpha))v_0(\alpha)\omega_0'(\alpha)\int_{v_0(\alpha)}^{v}\frac{1}{s^2(s^2-D(\alpha))^2}ds+1\right]\nonumber\\
    &:=& r_\alpha(v,\alpha).
\end{eqnarray}
In view of \eqref{vr}, the blowup of $v_{r}$ is either due to $v\rightarrow0$ or $r_{\alpha}\rightarrow0$. The following proposition describes the competition between these two limits.
\begin{Proposition}\label{pro:ralpha and v}
	If $v\rightarrow0$ and $r_{\alpha}\rightarrow0$ both happen, we have the following relation:
	\begin{itemize}
		\item If $\omega^{\prime}_{0}(\alpha)<0$, then $r_{\alpha}\rightarrow0$ before $v\rightarrow0$;
		\item If $\omega^{\prime}_{0}(\alpha)\geq0$, then $r_{\alpha}\rightarrow0$ and $v\rightarrow0$ simultaneously;
	\end{itemize}
\end{Proposition}
\begin{proof}
By the assumption, the two limits both happens. Since $v$ is strictly decreasing, we have $v_{0}(\alpha)>0$ in our consideration.

When $\omega^{\prime}_{0}(\alpha)<0$, we differentiate \eqref{r} with respect to $\alpha$ to obtain
	\begin{equation*}
		r_\alpha(t,\alpha)=2m\frac{\omega_0'(\alpha)(c^2-v^2)+2\omega_0(\alpha)vv_\alpha}{(\omega_0(\alpha)+c^2-v^2)^2}.
	\end{equation*}
	Thus, if there exists a finite time $t^*$ such that $v(t^*,\alpha)=0$ and by Proposition \ref{pro:3.3}, $D(\alpha)=\omega_0(\alpha)+c^2<0$, then it holds
	\begin{equation*}
		r_\alpha(t^*,\alpha)=\frac{2m\omega_0'(\alpha)c^2}{(\omega_0(\alpha)+c^2)^2}.
	\end{equation*}
Since $r_\alpha(0,\alpha)=1>0$, we can conclude that if $\omega_0'(\alpha)<0$, $r_\alpha\rightarrow0$ before $v\rightarrow0$.

When $\omega_0'(\alpha)=0$, by \eqref{ra2} we have
\begin{align*}
	r_{\alpha}=\frac{v(c^{2}-v^{2})}{v_{0}(\alpha)(c^{2}-v_{0}^{2}(\alpha))},
\end{align*}
which implies that the two limits happen simultaneously.

When $\omega^{\prime}_{0}(\alpha)>0$, the first term in the bracket on RHS of \eqref{ra2} is positive, which again implies that two limits happen simultaneously.
\end{proof}

\subsubsection{Dynamical behavior of $r_{\alpha}$}\label{subsubsec:3.2.3}
In order to study the relation between the behavior of $r_{\alpha}$ and initial data, we give an explicit classification of the initial data $v_0(\alpha)$ according to its magnitude, monotonicity and curvature.

We denote
\begin{equation*}
    V_+=\{v_0(\alpha)>0\},\ V_0=\{v_0(\alpha)=0\},\ V_-=\{v_0(\alpha)<0\},
\end{equation*}
\begin{equation*}
    \Omega_+=\{\omega_0'(\alpha)>0\},\ \Omega_0=\{\omega_0'(\alpha)=0\},\ \Omega_-=\{\omega_0'(\alpha)<0\},
\end{equation*}
\begin{equation*}
    D_+=\{D(\alpha)>0\},\ D_0=\{D(\alpha)=0\},\ D_-=\{D(\alpha)<0\},
\end{equation*}
\begin{equation*}
    E_1=V_-\cap \Omega_+\cap D_+, \ E_2=V_-\cap \Omega_+\cap D_0,\ E_3=V_-\cap \Omega_+\cap D_-, \ E_4=V_0,\
\end{equation*}
\begin{equation*}
    F_1=V_+\cap\Omega_+,\ F_2=V_-\cap\Omega_-,\ F_3=\Omega_0\cap (V_-\cup D_+\cup D_0),\
\end{equation*}
\begin{equation*}
    G_1=V_+\cap\Omega_0\cap D_-,\ G_2=V_+\cap\Omega_-.
\end{equation*}

In what follows, we analyze the behavior of $r_{\alpha}$
by considering eight distinct cases.

\textbf{Case 1}: $(\alpha, v_0(\alpha), v'_0(\alpha))\in E_1$, that is, $\omega'_0(\alpha)>0$,  $v_0(\alpha)<0$ and $D(\alpha)>0$.

By Proposition \ref{Prop1},
when $D(\alpha)>0$ and $v_0(\alpha)<0$, we get $-c<v_0(\alpha)<-\sqrt{\frac{2m}{\alpha}}c$.

For simplicity, we note $D(\alpha)=a^2$, by \eqref{ra2}, we have
\begin{equation*}
	\frac{1}{s^2(s^2-a^2)^2}=\frac{1}{a^4s^2}+\frac{3}{4a^5(s+a)}-\frac{3}{4a^5(s-a)}+\frac{1}{4a^4(s-a)^2}+\frac{1}{4a^4(s+a)^2},
\end{equation*}
so
\begin{equation}\label{Dis1}
	\begin{aligned}
		&\int_{v_0(\alpha)}^{v}\frac{ds}{s^2(s^2-D(\alpha))^2}=\int_{v_0(\alpha)}^{v}\frac{ds}{s^2(s^2-a^2)^2}=\left. \left[-\frac{1}{a^4s}-\frac{3}{4a^5}\ln \frac{s-a}{s+a}-\frac{s}{2a^4(s^2-a^2)}\right]\right|^{v}_{v_0(\alpha)}\\
		=&\frac{1}{D^2(\alpha)}\left(\frac{1}{v_0(\alpha)}-\frac{1}{v}+\frac{v_0(\alpha)}{2(v_0^2(\alpha)-D(\alpha))}-\frac{v}{2(v^2-D(\alpha))}\right)\\
		&+\frac{3}{4D^{\frac{5}{2}}(\alpha)}\ln \frac{\left(v_0(\alpha)-\sqrt{D(\alpha)}\right)\left(v+\sqrt{D(\alpha)}\right)}{\left(v_0(\alpha)+\sqrt{D(\alpha)}\right)\left(v-\sqrt{D(\alpha)}\right)}\\
        <&0,
	\end{aligned}
\end{equation}
since $v<v_0(\alpha)$ by the monotonicity of $v$ along the characteristic.

Denote
$$Dis(\alpha)=2m(c^2-v_0^2(\alpha))v_0(\alpha)\int_{v_0(\alpha)}^{-c}\frac{ds}{s^2(s^2-D(\alpha))^2}>0,$$
by \eqref{Dis1}, when $D(\alpha)>0$, we have
\begin{align*}
	Dis(\alpha)=&\frac{2m(c^2-v_0^2(\alpha))}{D^2(\alpha)}+\frac{2m(c^2-v_0^2(\alpha))v_0(\alpha)}{D^2(\alpha)c}+\frac{m(\alpha-2m)v_0^2(\alpha)}{2mD^2(\alpha)}+\frac{mc(\alpha-2m)v_0(\alpha)}{\alpha D^2(\alpha)}\\
	&+\frac{3m(c^2-v_0^2(\alpha))v_0(\alpha)}{2D^{\frac{5}{2}}(\alpha)}\ln \frac{\left(v_0(\alpha)-\sqrt{D(\alpha)}\right)\left(c-\sqrt{D(\alpha)}\right)}{\left(v_0(\alpha)+\sqrt{D(\alpha)}\right)\left(c+\sqrt{D(\alpha)}\right)}
\end{align*}

If
\begin{equation*}
	\omega_0'(\alpha)>\frac{1}{Dis(\alpha)},
\end{equation*}
and then it is easy to see that $\lim\limits_{t\rightarrow\infty}r_{\alpha}(t,\alpha)<0$, combining $r_{\alpha}(0,\alpha)=1$ and Intermediate Value Theorem for continuous functions, there must exist a finite time $0<t'<+\infty$ such that $r_\alpha(t')=0$. If
\begin{equation*}
	\omega_0'(\alpha)\leq\frac{1}{Dis(\alpha)},
\end{equation*}
then $r_\alpha>0$ for any finite time $t\in[0,+\infty)$.
\begin{remark}
	When $v_0(\alpha)\to \left(-\sqrt{\frac{2m}{\alpha}}c\right)^-$, we have $D(\alpha)\to 0^+$, which implies that $Dis(\alpha)=O\left(\left(v_0(\alpha)+\sqrt{\frac{2m}{\alpha}}c\right)^{-2} \right) \to \infty$ and $\frac{1}{Dis(\alpha)}=O\left(\left(v_0(\alpha)+\sqrt{\frac{2m}{\alpha}}c\right)^{2}\right) \to 0^+$.

	When $v_0(\alpha)\to -c$, by the Taylor expansion, we obtain
	\begin{equation*}
		Dis(\alpha)=2m(c^2-v_0^2(\alpha))v_0(\alpha)\int_{v_0(\alpha)}^{-c}\frac{ds}{s^2(s^2-D(\alpha))^2}=\frac{(\alpha-2m)^2}{2\alpha c^2}+o(v_0(\alpha)+c).
	\end{equation*}
	and
	\begin{equation*}
		\frac{1}{Dis(\alpha)}\rightarrow\frac{2\alpha c^2}{(\alpha-2m)^2 }.
	\end{equation*}
\end{remark}

\textbf{Case 2}: $(\alpha, v_0(\alpha), v'_0(\alpha))\in E_2$, that is, $\omega'_0(\alpha)>0$,  $v_0(\alpha)<0$ and $D(\alpha)=0$.

By Proposition \ref{Prop2}, when $D(\alpha)=0$ and $v_0(\alpha)<0$ then it holds that $v_0(\alpha)=-\sqrt{\frac{2m}{\alpha}}c$, and we get
\begin{equation*}
	\int_{v_0(\alpha)}^{v}\frac{ds}{s^2(s^2-D(\alpha))^2}=\int_{v_0(\alpha)}^{v}\frac{ds}{s^6}=\left.-\frac{1}{5s^5}\right|^{v}_{v_0(\alpha)}=-\frac{\left(\frac{\alpha}{2m}\right)^{\frac{5}{2}}}{5c^5}-\frac{1}{5v^5},
\end{equation*}
and
$$Dis(\alpha)=2m(c^2-v_0^2(\alpha))v_0(\alpha)\int_{v_0(\alpha)}^{-c}\frac{ds}{s^6}=\frac{(\alpha-2m)\left(\alpha^{5/2}-(2m)^{5/2}\right)}{10mc^2\alpha^{3/2}}.$$
Thus, if
\begin{equation*}
	\omega_0'(\alpha)>\frac{1}{Dis(\alpha)}=\frac{10mc^2\alpha^{3/2}}{(\alpha-2m)\left(\alpha^{5/2}-(2m)^{5/2}\right)  },
\end{equation*}
by the similar discussions of \textbf{Case 1}, there exists a finite time $t'$ such that $r_\alpha(t')=0$ according to \eqref{ra2}. If
\begin{equation*}
	\omega_0'(\alpha)\leq\frac{10mc^2\alpha^{3/2}}{(\alpha-2m)\left(\alpha^{5/2}-(2m)^{5/2}\right) },
\end{equation*}
then $r_\alpha>0$ for any finite time $t\in[0,+\infty)$.

\textbf{Case 3}: $(\alpha, v_0(\alpha), v'_0(\alpha))\in E_3$, that is, $\omega'_0(\alpha)>0$,  $v_0(\alpha)<0$ and $D(\alpha)<0$.

By Proposition \ref{pro:3.3}, when $D(\alpha)<0$ and $v_0(\alpha)<0$, we have $-\sqrt{\frac{2m}{\alpha}}c<v_0(\alpha)<0$. We denote $D(\alpha)=-b^2$ for simplicity, we can obtain
\begin{equation*}
	\frac{1}{s^2(s^2+b^2)^2}=\frac{1}{b^4s^2}-\frac{1}{b^4(s^2+b^2)}-\frac{1}{b^2(s^2+b^2)^2},
\end{equation*}
so
\begin{equation*}
	\begin{aligned}
		\int_{v_0(\alpha)}^{v}\frac{ds}{s^2(s^2-D(\alpha))^2}=&\int_{v_0(\alpha)}^{v}\frac{dx}{s^2(s^2+b^2)^2}=\left. \left[-\frac{1}{b^4s}-\frac{3}{2b^5}\arctan \frac{s}{b}-\frac{s}{2b^4(s^2+b^2)}\right]\right|^{v}_{v_0(\alpha)}\\
		=&\frac{1}{D^2(\alpha)}\left(\frac{1}{v_0(\alpha)}-\frac{1}{v}+\frac{v_0(\alpha)}{2(v_0^2(\alpha)-D(\alpha))}-\frac{v}{2(v^2-D(\alpha))}\right)\\
		&+\frac{3}{2(-D(\alpha))^{\frac{5}{2}}}\left(\arctan \frac{v_0(\alpha)}{(-D(\alpha))^{\frac{1}{2}}}-\arctan \frac{v}{(-D(\alpha))^{\frac{1}{2}}}\right),
	\end{aligned}
\end{equation*}
and in this case
\begin{align*}
	Dis(\alpha)=&\frac{2m(c^2-v_0^2(\alpha))}{D^2(\alpha)}+\frac{2m(c^2-v_0^2(\alpha))v_0(\alpha)}{D^2(\alpha)c}+\frac{m(\alpha-2m)v_0^2(\alpha)}{2mD^2(\alpha)}+\frac{mc(\alpha-2m)v_0(\alpha)}{\alpha D^2(\alpha)}\\
	&+\frac{3m(c^2-v_0^2(\alpha))v_0(\alpha)}{(-D(\alpha))^{\frac{5}{2}}}\left(\arctan \frac{v_0(\alpha)}{(-D(\alpha))^{\frac{1}{2}}}+\arctan \frac{c}{(-D(\alpha))^{\frac{1}{2}}}\right).
\end{align*}
We can obtain from \eqref{ra2} and similar discussions as \textbf{Case 1}, if
\begin{equation*}
	\omega_0'(\alpha)>\frac{1}{Dis(\alpha)},
\end{equation*}
there exists a time $t'$ such that $r_\alpha(t')=0$. If
\begin{equation*}
	\omega_0'(\alpha)\leq\frac{1}{Dis(\alpha)},
\end{equation*}
then $r_\alpha>0$ for all $t>0$.
\begin{remark}\label{0-}
	As $v_0(\alpha)\to \left(-\sqrt{\frac{2m}{\alpha}}c\right)^+$, we have $D(\alpha)\to 0^-$, which implies that $$Dis(\alpha)=O\left(\left(v_0(\alpha)+\sqrt{\frac{2m}{\alpha}}c\right)^{-2}\right) \to \infty$$ and $$\frac{1}{Dis(\alpha)}=O\left(\left(v_0(\alpha)+\sqrt{\frac{2m}{\alpha}}c\right)^{2}\right) \to 0^+.$$
When $v_0(\alpha)\to 0^-$, we have
	\begin{equation*}
		\begin{aligned}
			Dis(\alpha)=&\frac{2mc^2}{(\alpha-2m)^2}+O(v_0(\alpha)),\quad \Rightarrow\quad \lim\limits_{v_0(\alpha)\to 0^-}\frac{1}{Dis(\alpha)}=\frac{2mc^2}{(\alpha-2m)^2}.
		\end{aligned}
		\end{equation*}
\end{remark}
\textbf{Case 4}: $(\alpha, v_0(\alpha), v'_0(\alpha))\in E_4$, that is, $v_0(\alpha)=0$.
When $v_0(\alpha) = 0$, we need to first confirm whether \eqref{ra2} is well-defined as $v_0(\alpha)$ approaches zero. When $v_0(\alpha)$ is sufficiently close to 0, then there must be $D(\alpha)<0$. Taking \eqref{w0'} into the conclusion in \textbf{Case 3}, we have
\begin{equation}\label{ra-}
	\begin{aligned}
		&r_\alpha(v;v_0(\alpha),\alpha)=\frac{v(c^2-v^2)}{v_0(\alpha)(c^2-v^2_0(\alpha))}\left[-2m(c^2-v_0^2(\alpha))v_0(\alpha)\omega_0'(\alpha)\int_{v_0(\alpha)}^{v}\frac{1}{s^2(s^2-D(\alpha))^2}ds+1\right]\\
		=&-2m\omega_0'(\alpha)(c^2-v^2)v\left[\frac{1}{D^2(\alpha)}\left(\frac{1}{v_0(\alpha)}-\frac{1}{v}+\frac{v_0(\alpha)}{2(v_0^2(\alpha)-D(\alpha))}-\frac{v}{2(v^2-D(\alpha))}\right)\right.\\
		&\left.+\frac{3}{2(-D(\alpha))^{\frac{5}{2}}}\left(\arctan \frac{v_0(\alpha)}{(-D(\alpha))^{\frac{1}{2}}}-\arctan \frac{v}{(-D(\alpha))^{\frac{1}{2}}}\right)\right]+\frac{v(c^2-v^2)}{v_0(\alpha)(c^2-v^2_0(\alpha))}\\
		=&-\frac{4m\alpha(\alpha-2m)v'_0(\alpha)(c^2-v^2)v}{(\alpha v_0^2(\alpha)-2mc^2)^2}+\frac{\left[(\alpha^2-4m^2)v_0^3(\alpha)-4mc^2v_0(\alpha)(\alpha-2m)\right](c^2-v^2)v}{(\alpha v_0^2(\alpha)-2mc^2)^2(c^2-v^2_0(\alpha))}\\
		&+\frac{2m\omega_0'(\alpha)(c^2-v^2)}{D^2(\alpha)}-\frac{m\omega_0'(\alpha)(c^2-v^2)vv_0(\alpha)}{D^2(\alpha)(v_0^2(\alpha)-D(\alpha))}+\frac{m\omega_0'(\alpha)(c^2-v^2)v^2}{D^2(\alpha)(v^2-D(\alpha))}\\
		&-\frac{3m\omega_0'(\alpha)(c^2-v^2)v}{(-D(\alpha))^{\frac{5}{2}}}\arctan \frac{v_0(\alpha)}{(-D(\alpha))^{\frac{1}{2}}}+\frac{3m\omega_0'(\alpha)(c^2-v^2)v}{(-D(\alpha))^{\frac{5}{2}}}\arctan \frac{v}{(-D(\alpha))^{\frac{1}{2}}}.
	\end{aligned}
\end{equation}
Let $v_0(\alpha)\to 0$,  we obtain
\begin{equation}\label{ra0}
	\begin{aligned}
		r_\alpha(v;0,\alpha)=&-\frac{\alpha (\alpha-2m)(c^2-v^2)v'_0(\alpha)v}{mc^4}+\frac{c^2-v^2}{c^2}+\frac{(c^2-v^2)v^2}{2c^2(v^2+\frac{2mc^2}{\alpha-2m})}\\
		&+\frac{3\sqrt{2}(\alpha-2m)^{\frac{1}{2}}(c^2-v^2)v}{2m^{\frac{1}{2}}c^3}\arctan\frac{(\alpha-2m)^{\frac{1}{2}}v}{\sqrt{2}m^{\frac{1}{2}}c}.
	\end{aligned}
\end{equation}
It is also easy to verify that $\lim\limits_{v\to 0}r_\alpha(v;0,\alpha)=1$ by \eqref{ra0}.

Synthesizing the preceding arguments, we proved that the singular integral in \eqref{ra1} is well-defined even $v_0(\alpha) = 0$. Moreover, it can be expressed by \eqref{ra0}. Furthermore, by the local well-posedness of the solution, there exists a time $\overline{t} > 0$ such that $r_\alpha(t,\alpha) > 0$ for all $t \in (0,\overline{t})$. 

Next, we use \eqref{ra0} to estimate the behavior of $r_\alpha$ with respect to $v$ from time $\overline{t}$, in this case, we have $0<r_\alpha<1$,   $v(\overline{t})<0$, and then the velocity $v$ lies in $(-c, v(\overline{t})]$. Direct calculations give us
\begin{equation*}
	\begin{aligned}
		r_\alpha(v;0,\alpha)=&-\frac{\alpha(\alpha-2m) (c^2-v^2)v'_0(\alpha)v}{mc^4}+\frac{c^2-v^2}{c^2}+\frac{(c^2-v^2)v^2}{2c^2(v^2+\frac{2mc^2}{\alpha-2m})}\\
		&+\frac{3\sqrt{2}(\alpha-2m)^{\frac{1}{2}}(c^2-v^2)v}{2m^{\frac{1}{2}}c^3}\arctan\frac{(\alpha-2m)^{\frac{1}{2}}v}{\sqrt{2}m^{\frac{1}{2}}c}\\
		:=&\frac{-\alpha(\alpha-2m)v(c^2-v^2)}{mc^4}\left(v'_0(\alpha)-\tau(v,\alpha)\right),
	\end{aligned}
\end{equation*}
where
\begin{align*}
	\tau(v,\alpha)=&\frac{mc^2}{\alpha(\alpha-2m)v}+\frac{mc^2v}{2\alpha\left[(\alpha-2m)v^2+2mc^2\right]}+\frac{3\sqrt{2}c}{2\alpha}\sqrt{\frac{m}{\alpha-2m}}\arctan\left(\sqrt{\frac{\alpha-2m}{2m}}\frac{v}{c}\right).
\end{align*}
Since $v\in\left(-c, v(\overline{t})\right]$, we have $$\frac{-\alpha(\alpha-2m)v(c^2-v^2)}{mc^4}>0.$$

Therefore, we only need to consider whether there exists a $v\in \left(-c, v(\overline{t})\right]$ such that $v'_0(\alpha)-\tau(v,\alpha)=0$.

We need to study the extreme value of $\tau(v,\alpha)$, it is easy to get
\begin{align*}
	\frac{d\tau(v,\alpha)}{dv}=&-\frac{mc^2}{\alpha(\alpha-2m)v^2}+\frac{mc^2}{2\alpha}\frac{2mc^2-(\alpha-2m)v^2}{[(\alpha-2m)v^2+2mc^2]^2}+\frac{3mc^2}{\alpha\left[(\alpha-2m)v^2+2mc^2\right]}\\
	=&\frac{mc^2}{2\alpha}\frac{3(\alpha-2m)^2v^4+6mc^2(\alpha-2m)v^2-8m^2c^4}{(\alpha-2m)[(\alpha-2m)v^2+2mc^2]^2v^2}.
\end{align*}
Let $\frac{d\tau(v,\alpha)}{dv}=0$, and solve $$3(\alpha-2m)^2v^4+6mc^2(\alpha-2m)v^2-8m^2c^4=0,$$
we have $v^2=\frac{(\sqrt{33}-3)m}{3(\alpha-2m)}c^2$.

\textbf{Case 4.1}: When $\frac{(\sqrt{33}-3)m}{3(\alpha-2m)}\geq1$, or $2m<\alpha\leq\frac{\sqrt{33}+3}{3}m$
	, then $\tau(v,\alpha)$ is decreasing on $\left(-c, v(\overline{t})\right]$, then
	\begin{equation*}
		\max\limits_{v\in\left(-c, v(\overline{t})\right]}\tau(v,\alpha)=\tau(-c,\alpha)=-\frac{mc(3\alpha-2m)}{2\alpha^2(\alpha-2m)}-\frac{3\sqrt{2}c}{2\alpha}\sqrt{\frac{m}{\alpha-2m}}\arctan\sqrt{\frac{\alpha-2m}{2m}}.
	\end{equation*}

    If $$v'_0(\alpha)\geq-\frac{mc(3\alpha-2m)}{2\alpha^2(\alpha-2m)}-\frac{3\sqrt{2}c}{2\alpha}\sqrt{\frac{m}{\alpha-2m}}\arctan\sqrt{\frac{\alpha-2m}{2m}},$$
	then for any fimite time $t>\overline{t}$ and  $v(t)\in\left(-c, v(\overline{t})\right]$, $v'_0(\alpha)-\tau(v,\alpha)>0$, thus, $r_\alpha(v;0,\alpha)>0$ for all $t>0$.

    And if
	$$v'_0(\alpha)<-\frac{mc(3\alpha-2m)}{2\alpha^2(\alpha-2m)}-\frac{3\sqrt{2}c}{2\alpha}\sqrt{\frac{m}{\alpha-2m}}\arctan\sqrt{\frac{\alpha-2m}{2m}},$$ then there exists a finite time $t'$ and $v(t')\in \left(-c, v(\bar{t})\right]$ such that $r_\alpha(v(t');0,\alpha)=0$.

\textbf{Case 4.2}: When $0<\frac{(\sqrt{33}-3)m}{3(\alpha-2m)}<1$, or $\alpha>\frac{\sqrt{33}+3}{3}m$, denote $$\hat{v}=-\sqrt{\frac{(\sqrt{33}-3)m}{3(\alpha-2m)}}c,$$
	then $\tau(v,\alpha)$ is increasing on $\left(-c, \hat{v}\right]$ and decreasing on $\left(\hat{v},v(\overline{t})\right]$, then
	\begin{equation*}
		\max\limits_{v\in\left(-c, v(\overline{t})\right]}\tau(v,\alpha)=\tau(\hat{v},\alpha)=-\frac{c}{2\alpha}\sqrt{\frac{m}{\alpha-2m}}\left(\sqrt{\frac{\sqrt{33}+3}{2}}+\frac{(\sqrt{33}-3)^{\frac{3}{2}}}{8\sqrt{3}}+3\sqrt{2}\arctan\sqrt{\frac{\sqrt{33}-3}{6}}\right).
	\end{equation*}
	If $$v'_0(\alpha)>-\frac{c}{2\alpha}\sqrt{\frac{m}{\alpha-2m}}\left(\sqrt{\frac{\sqrt{33}+3}{2}}+\frac{(\sqrt{33}-3)^{\frac{3}{2}}}{8\sqrt{3}}+3\sqrt{2}\arctan\sqrt{\frac{\sqrt{33}-3}{6}}\right),$$
	then for any fimite time $t>\overline{t}$ and  $v(t)\in\left(-c, v(\overline{t})\right]$, $v'_0(\alpha)-\tau(v,\alpha)>0$, thus $r_\alpha(v;0,\alpha)>0$ for all $t>0$.

     And if
	$$v'_0(\alpha)\leq-\frac{c}{2\alpha}\sqrt{\frac{m}{\alpha-2m}}\left(\sqrt{\frac{\sqrt{33}+3}{2}}+\frac{(\sqrt{33}-3)^{\frac{3}{2}}}{8\sqrt{3}}+3\sqrt{2}\arctan\sqrt{\frac{\sqrt{33}-3}{6}}\right),$$ then there exists a finite time $t'$ and $v(t')\in \left(-c, v(\bar{t})\right]$ such that $r_\alpha(v(t');0,\alpha)=0$.

\textbf{Case 5}: $(\alpha, v_0(\alpha), v'_0(\alpha))\in F_1\cup F_2$, that is, $\omega_0'(\alpha)v_0(\alpha)>0$.

In this case, if $v\neq0$ at any finite time, which is equivalent to that $\frac{v(t)}{v_0(\alpha)}>0$ for any $t>0$, from \eqref{ra2}, we obtain
\begin{equation*}
	-2m(c^2-v_0^2(\alpha))v_0(\alpha)\omega_0'(\alpha)\int_{v_0(\alpha)}^{v}\frac{1}{s^2(s^2-D(\alpha))^2}ds+1>1
\end{equation*}
and
\begin{equation*}
	r_\alpha>\frac{v(t)(c^2-v^2(t))}{v_0(\alpha)(c^2-v^2_0(\alpha))}>0.
\end{equation*}

If there exists $t_0>0$ such that $v(t_0,\alpha)=0$, then by Section \ref{subsection3.2.1}, the initial data need to satisfy $v_0(\alpha)>0$, $\omega_0'(\alpha)>0$ and $D(\alpha)<0$. By \eqref{ra-}, we have
\begin{equation*}
	\lim\limits_{t\to t_0}r_\alpha=\lim\limits_{v\to 0}r_\alpha=\frac{2mc^2\omega_0'(\alpha)}{D^2(\alpha)}>0.
\end{equation*}

\textbf{Case 6}: $(\alpha, v_0(\alpha), v'_0(\alpha))\in F_3$, that is, $\omega'_0(\alpha)=0,\ v_0(\alpha)<0\mbox{ or }D(\alpha)=\omega_0(\alpha)+c^2\geq0$.

If $\omega_0'(\alpha)=0$, we have
$$r_\alpha=\frac{v(t)(c^2-v^2(t))}{v_0(\alpha)(c^2-v^2_0(\alpha))}.$$
Given $v_0(\alpha)<0$, we have $-c<v<v_0(\alpha)<0$, this leads to $r_\alpha>0$.

If $D(\alpha)\geq0$, according to Proposition \ref{Prop1} and Proposition \ref{Prop2}, we can also obtain that $v\neq0$ and the signs of $v$ and $v_0(\alpha)$ remain identical at any finite time, thus $r_\alpha>0$.

\textbf{Case 7}: $(\alpha, v_0(\alpha), v'_0(\alpha))\in G_1$, that is, $\omega'_0(\alpha)=0,\ v_0(\alpha)>0 \mbox{ and} \ D(\alpha)<0$.

According to Proposition \ref{pro:ralpha and v}, in this case, there exists a time $0<t'<+\infty$, such that $r_\alpha$ and $v$ approach 0 as $t \to t'$.

\textbf{Case 8}: $(\alpha, v_0(\alpha), v'_0(\alpha))\in G_2$, that is, $\omega'_0(\alpha)<0,\ v_0(\alpha)>0$.

 In this case, we split into three cases according to the symbol of $D(\alpha)$ to discuss the dynamics of $r_{\alpha}$.

\textbf{Case 8.1}: For $D(\alpha)>0$, we have $v_0(\alpha)\geq v>\sqrt{D(\alpha)}$, and $v\to \sqrt{D(\alpha)}$ as $t \to \infty$.

According to the conclusion in \textbf{Case 1}, we have
\begin{equation*}
	\begin{aligned}
		\lim\limits_{v\to \sqrt{D(\alpha)}}\int_{v_0(\alpha)}^{v}\frac{ds}{s^2(s^2-D(\alpha))^2}=&\lim\limits_{v\to \sqrt{D(\alpha)}}\left[\frac{1}{D^2(\alpha)}\left(\frac{1}{v_0(\alpha)}-\frac{1}{v}+\frac{v_0(\alpha)}{2(v_0^2(\alpha)-D(\alpha))}-\frac{v}{2(v^2-D(\alpha))}\right)\right.\\
		&\left.+\frac{3}{4D^{\frac{5}{2}}(\alpha)}\ln \frac{\left(v_0(\alpha)-\sqrt{D(\alpha)}\right)\left(v+\sqrt{D(\alpha)}\right)}{\left(v_0(\alpha)+\sqrt{D(\alpha)}\right)\left(v-\sqrt{D(\alpha)}\right)}\right]\\
		=&-\infty.
	\end{aligned}
\end{equation*}
and
\begin{equation*}
	r_\alpha=\frac{v(c^2-v^2)}{v_0(\alpha)(c^2-v^2_0(\alpha))}\left[-2m(c^2-v_0^2(\alpha))v_0(\alpha)\omega_0'(\alpha)\int_{v_0(\alpha)}^{v}\frac{1}{s^2(s^2-D(\alpha))^2}ds+1\right]\to -\infty
\end{equation*}
Consequently, there must exist a finite time $t'$ such that $r_\alpha(t',\alpha)=0$.

\textbf{Case 8.2}: For $D(\alpha)=0$, we have $v_{D}(\alpha)=v_0(\alpha)\geq v>0$, and $v\to 0$ as $t \to \infty$.

According to the conclusion in \textbf{Case 2}, we have
\begin{equation*}
	\lim\limits_{v\to 0}\int_{v_0(\alpha)}^{v}\frac{ds}{s^2(s^2-D(\alpha))^2}=\lim\limits_{v\to 0}\left(\frac{1}{5v_0^5(\alpha)}-\frac{1}{5v^5}\right)=-\infty.
\end{equation*}
Then there must exist a time $t'$ such that $r_\alpha(t',\alpha)=0$.

\textbf{Case 8.3}: For $D(\alpha)<0$, Proposition \ref{Prop4} and Proposition \ref{pro:ralpha and v} give a blowup time $t^*=-\phi(v_0(\alpha))$ at which $v(t^*)=0$. Moreover, there exists a prior vanishing time $t'<t^*$ such that $r_\alpha\rightarrow0$ as $t\rightarrow t'$.

\begin{remark}[Taylor expansion near the blowup time $t'$]
When $r_\alpha(t',\alpha)=0$ and if $\omega_0'(\alpha)\neq0$, then by \eqref{ra}, we have
\begin{equation*}
	\left.\frac{d r_\alpha}{d t}\right|_{t=t'}=\frac{\left(1-\frac{2m}{r(t')}\right)^2\omega_0'(\alpha)}{2cv(t')}\neq0,
\end{equation*}
so we can expend $r_\alpha(t,\alpha)$ at $t'$ by Taylor expansion.

For any $\varepsilon>0$, there exists $\delta_1>0$ such that when $|t-t'|<\delta_1$, we have
\begin{equation}
	\left|r_\alpha(t,\alpha)-\frac{\left(1-\frac{2m}{r(t')}\right)^2\omega_0'(\alpha)}{2cv(t')}(t-t')\right|<\varepsilon.
\end{equation}

If $\omega_0'(\alpha)=0$, then \eqref{X1} turns into
\begin{equation*}
	X=\frac{1-\frac{2m}{r}}{c}\left(\frac{m\omega_0(\alpha)}{r^2v}+\frac{2v}{r}\right).
\end{equation*}
Therefore, in this case, we only need to analyze the blow-up order of $\rho$ through $v$, instead of the Taylor expansion in $r_\alpha$.
\end{remark}
\begin{remark}
	In the above derivation, we denote the vanishing time of each $r_\alpha$ uniformly by $t'$. In fact, these vanishing times are usually different from each other. The notation $t'$ is adopted merely to distinguish them from the time $t^*$ defined by $v(t^*)=0$.
\end{remark}
\subsection{Solution formula and analysis of $\frac{\partial \rho}{\partial r}$}\label{subsec3}
For a classical solution, if $\rho\in C^0$ and $v\in C^1$, we need further to study the behavior of $\rho_r:=\frac{\partial  \rho}{\partial r}$.
Differentiate \eqref{ans1} with respect to $r$, we have
\begin{align}\label{ans2}
\frac{\partial  \rho}{\partial r}=\frac{\partial  \rho_0(\alpha)}{\partial  r}\cdot e^{-\int_{0}^{t} Xdt}+\frac{\partial  \left(e^{-\int_{0}^{t} Xdt}\right)}{\partial  r}\cdot \rho_0(\alpha)\nonumber=\left(\frac{\partial  \rho_0(\alpha)}{\partial \alpha}\cdot\frac{1}{r_\alpha}-\frac{\rho_0(\alpha)}{c}\int_{0}^{t}Ydt \right)\cdot e^{-\int_{0}^{t} Xdt              },
\end{align}
where
\begin{equation}\label{Y}
	Y=\left(f'(r)+\frac{2f(r)}{r}\right)v_r+f(r)v_{rr}+\frac{2rf'(r)v-2f(r)v}{r^2}	.
\end{equation}
Differentiate \eqref{vr} with respect to $r$, we have
\begin{equation}\label{vrr}
	\begin{aligned}
		v_{rr}=\frac{v_{\alpha\alpha}r_\alpha-r_{\alpha\alpha}v_\alpha}{r_\alpha^3}.
	\end{aligned}
\end{equation}
In Section \ref{subsubsec:3.2.3}, we have already given the expression for $r_\alpha$ under different initial conditions $D(\alpha )$. Differentiating $r_\alpha$ under different initial conditions $D(\alpha)$ with respect to $\alpha$, we can obtain the following proposition
\begin{Proposition}
    If $\rho(t,r)\in C^0([0,t^*]\times (2m,+\infty))$ and $v(t,r)\in C^1([0,t^*]\times (2m,+\infty))$ for $0<t^*<+\infty$, then we have $\rho(t,r)\in C^1([0,t^*]\times (2m,+\infty))$ and $v(t,r)\in C^2([0,t^*]\times (2m,+\infty))$.
\end{Proposition}
\begin{proof}
    We prove the above proposition in three cases according to the sign of $D(\alpha)$.

\textbf{Case 1}: $D(\alpha)>0$. Direct calculations give
\begingroup
\allowdisplaybreaks
\begin{align*}
    r_{\alpha\alpha}(v,\alpha)=&-\frac{8m(\alpha-m)v'_0(\alpha)(c^2-v^2)v}{(\alpha v_0^2(\alpha)-2mc^2)^2}-\frac{4m\alpha(\alpha-2m)v''_0(\alpha)(c^2-v^2)v}{(\alpha v_0^2(\alpha)-2mc^2)^2}\\
    &-\frac{4m\alpha(\alpha-2m)v'_0(\alpha)(c^2-3v^2)v_\alpha}{(\alpha v_0^2(\alpha)-2mc^2)^2}+\frac{8m\alpha(\alpha-2m)v'_0(\alpha)(v_0^2(\alpha)+2\alpha v_0(\alpha)v'_0(\alpha))(c^2-v^2)v}{(\alpha v_0^2(\alpha)-2mc^2)^3}\\
    &+\frac{\left[2\alpha v_0^3(\alpha)+3(\alpha^2-4m^2)v_0^2(\alpha)v'_0(\alpha)-4mc^2v_0(\alpha)-4mc^2(\alpha-2m)v'_0(\alpha)\right](c^2-v^2)v}{(\alpha v_0^2(\alpha)-2mc^2)^2(c^2-v_0^2(\alpha))}\\
    &+\frac{\left[(\alpha^2-4m^2)v_0^3(\alpha)-4mc^2(\alpha-2m)v_0(\alpha)\right](c^2-3v^2)v_\alpha}{(\alpha v_0^2(\alpha)-2mc^2)^2(c^2-v_0^2(\alpha))}\\
    &-\frac{2\left[(\alpha^2-4m^2)v_0^3(\alpha)-4mc^2(\alpha-2m)v_0(\alpha)\right](c^2-v^2)(v_0^2(\alpha)+2\alpha v_0(\alpha)v'_0(\alpha))v}{(\alpha v_0^2(\alpha)-2mc^2)^3(c^2-v_0^2(\alpha))}\\
    &+\frac{2\left[(\alpha^2-4m^2)v_0^3(\alpha)-4mc^2(\alpha-2m)v_0(\alpha)\right](c^2-v^2)v_0(\alpha)v'_0(\alpha)v}{(\alpha v_0^2(\alpha)-2mc^2)^2(c^2-v_0^2(\alpha))}\\
    &+\frac{2m(c^2-v^2)\omega_0''(\alpha)}{D^2(\alpha)}-\frac{4m\omega_0'(\alpha)vv_\alpha}{D^2(\alpha)}-\frac{4m(c^2-v^2)\omega_0'(\alpha)D'(\alpha)}{D^3(\alpha)}\\
    &-\frac{m\omega_0''(\alpha)(c^2-v^2)v_0(\alpha)v}{D^2(\alpha)(v_0^2(\alpha)-D(\alpha))}-\frac{m\omega_0'(\alpha)(c^2-3v^2)v_0(\alpha)v_\alpha}{D^2(\alpha)(v_0^2(\alpha)-D(\alpha))}-\frac{m\omega_0'(\alpha)(c^2-v^2)v'_0(\alpha)v}{D^2(\alpha)(v_0^2(\alpha)-D(\alpha))}\\
    &+\frac{2m\omega_0'(\alpha)(c^2-v^2)v_0(\alpha)D'(\alpha)v}{D^3(\alpha)(v_0^2(\alpha)-D(\alpha))}+\frac{m\omega_0'(\alpha)(c^2-v^2)v_0(\alpha)(2v_0(\alpha)v'_0(\alpha)-D'(\alpha))v}{D^2(\alpha)(v_0^2(\alpha)-D(\alpha))^2}\\
    &+\frac{m\omega_0''(\alpha)(c^2-v^2)v^2}{D^2(\alpha)(v_0^2(\alpha)-D(\alpha))}+\frac{2m\omega_0'(\alpha)(c^2-2v^2)vv_\alpha}{D^2(\alpha)(v_0^2(\alpha)-D(\alpha))}-\frac{2m\omega_0'(\alpha)(c^2-v^2)D'(\alpha)v^2}{D^3(\alpha)(v_0^2(\alpha)-D(\alpha))}\\
    &-\frac{m\omega_0'(\alpha)(c^2-v^2)(2vv_\alpha-D'(\alpha))v^2}{D^2(\alpha)(v_0^2(\alpha)-D(\alpha))^2}-\frac{3m\omega_0''(\alpha)(c^2-v^2)}{2D^{\frac{5}{2}}(\alpha)}\ln \frac{\left(v_0(\alpha)-\sqrt{D(\alpha)}\right)\left(v+\sqrt{D(\alpha)}\right)}{\left(v_0(\alpha)+\sqrt{D(\alpha)}\right)\left(v-\sqrt{D(\alpha)}\right)}\\
    &-\frac{3m\omega_0'(\alpha)(c^2-3v^2)v_\alpha}{2D^{\frac{5}{2}}(\alpha)}\ln \frac{\left(v_0(\alpha)-\sqrt{D(\alpha)}\right)\left(v+\sqrt{D(\alpha)}\right)}{\left(v_0(\alpha)+\sqrt{D(\alpha)}\right)\left(v-\sqrt{D(\alpha)}\right)}\\
    &+\frac{15m\omega_0'(\alpha)(c^2-v^2)D'(\alpha)v}{4D^{\frac{7}{2}}(\alpha)}\ln \frac{\left(v_0(\alpha)-\sqrt{D(\alpha)}\right)\left(v+\sqrt{D(\alpha)}\right)}{\left(v_0(\alpha)+\sqrt{D(\alpha)}\right)\left(v-\sqrt{D(\alpha)}\right)}\\
    &-\frac{3m\omega_0'(\alpha)(c^2-v^2)v}{2D^{\frac{5}{2}}(\alpha)}\left(\frac{2\sqrt{D(\alpha)}v'_0(\alpha)-v_0(\alpha)D'(\alpha)}{v_0^2(\alpha)-D(\alpha)}\right),
\end{align*}
\endgroup
in which
\begin{equation}
    \omega_0''(\alpha)=\frac{4v_0(\alpha)v'_0(\alpha)+2\alpha\left[\left(v'_0(\alpha)\right)^2+v_0(\alpha)v''_0(\alpha)\right]}{\alpha-2m}-\frac{4\alpha v_0(\alpha)v'_0(\alpha)}{(\alpha-2m)^2}+\frac{4m(v_0^2(\alpha)-c^2)}{(\alpha-2m)^3}
\end{equation}
and $ D'(\alpha)=\omega_0'(\alpha)$.
As for $v_{\alpha\alpha}$, by the expression for $v_\alpha$ in Remark \ref{vad}, and differentiate them with respect to $\alpha$, we have
\begingroup
\allowdisplaybreaks
\begin{align*}
    v_{\alpha\alpha}=&-\frac{\omega_0''(\alpha)(c^2-v^2)v}{2D(\alpha)\omega_0(\alpha)}-\frac{\omega_0'(\alpha)(c^2-3v^2)v_\alpha}{2D(\alpha)\omega_0(\alpha)}+\frac{\omega_0'(\alpha)D'(\alpha)(c^2-v^2)v}{2D^2(\alpha)\omega_0(\alpha)}+\frac{(\omega_0'(\alpha))^2(c^2-v^2)v}{2D(\alpha)\omega_0^2(\alpha)}\\
    &+\frac{\omega_0''(\alpha)v_0(\alpha)(c^2-v^2)(D(\alpha)-v^2)^2}{2D(\alpha)\omega_0(\alpha)(D(\alpha)-v_0^2(\alpha))^2}+\frac{\omega_0'(\alpha)v'_0(\alpha)(c^2-v^2)(D(\alpha)-v^2)^2}{2D(\alpha)\omega_0(\alpha)(D(\alpha)-v_0^2(\alpha))^2}\\
    &-\frac{\omega_0'(\alpha)v_0(\alpha)(D(\alpha)-v^2)(D(\alpha)+2c^2-3v^2)vv_\alpha}{D(\alpha)\omega_0(\alpha)(D(\alpha)-v_0^2(\alpha))^2}-\frac{\omega_0'(\alpha)v_0(\alpha)(c^2-v^2)(D(\alpha)-v^2)^2D'(\alpha)}{2D^2(\alpha)\omega_0(\alpha)(D(\alpha)-v_0^2(\alpha))^2}\\
    &-\frac{(\omega_0'(\alpha))^2v_0(\alpha)(c^2-v^2)(D(\alpha)-v^2)^2}{2D(\alpha)\omega_0^2(\alpha)(D(\alpha)-v_0^2(\alpha))^2}-\frac{\omega_0'(\alpha)v_0(\alpha)(c^2-v^2)(D(\alpha)-v^2)^2D'(\alpha)}{D(\alpha)\omega_0(\alpha)(D(\alpha)-v_0^2(\alpha))^3}\\
    &+\frac{2\omega_0'(\alpha)v_0^2(\alpha)v'_0(\alpha)(c^2-v^2)(D(\alpha)-v^2)^2}{D(\alpha)\omega_0(\alpha)(D(\alpha)-v_0^2(\alpha))^3}+\frac{3\omega_0''(\alpha)(c^2-v^2)(v^2-D(\alpha))v}{4D^2(\alpha)\omega_0(\alpha)}\\
    &+\frac{3\omega_0'(\alpha)(c^2-3v^2)(v^2-D(\alpha))v_\alpha}{4D^2(\alpha)\omega_0(\alpha)}+\frac{3\omega_0'(\alpha)(c^2-v^2)v^2v_\alpha}{2D^2(\alpha)\omega_0(\alpha)}-\frac{3\omega_0'(\alpha)D'(\alpha)(c^2-v^2)v}{4D^2(\alpha)\omega_0(\alpha)}\\
    &-\frac{3\omega_0'(\alpha)D'(\alpha)(c^2-v^2)(v^2-D(\alpha))v}{2D^3(\alpha)\omega_0(\alpha)}-\frac{3(\omega_0'(\alpha))^2(c^2-v^2)(v^2-D(\alpha))v}{4D^2(\alpha)\omega_0^2(\alpha)}\\
    &-\frac{3\omega_0''(\alpha)v_0(\alpha)(c^2-v^2)(v^2-D(\alpha))^2}{4D^2(\alpha)\omega_0(\alpha)(v^2_0(\alpha)-D(\alpha))}-\frac{3\omega_0'(\alpha)v'_0(\alpha)(c^2-v^2)(v^2-D(\alpha))^2}{4D^2(\alpha)\omega_0(\alpha)(v^2_0(\alpha)-D(\alpha))}\\
    &+\frac{3\omega_0'(\alpha)v_0(\alpha)(v^2-D(\alpha))^2vv_\alpha}{2D^2(\alpha)\omega_0(\alpha)(v^2_0(\alpha)-D(\alpha))}-\frac{3\omega_0'(\alpha)v_0(\alpha)(c^2-v^2)(v^2-D(\alpha))vv_\alpha}{D^2(\alpha)\omega_0(\alpha)(v^2_0(\alpha)-D(\alpha))}\\
    &+\frac{3\omega_0'(\alpha)D'_0(\alpha)v_0(\alpha)(c^2-v^2)(v^2-D(\alpha))^2}{2D^2(\alpha)\omega_0(\alpha)(v^2_0(\alpha)-D(\alpha))}+\frac{3\omega_0'(\alpha)D'(\alpha)v_0(\alpha)(c^2-v^2)(v^2-D(\alpha))^2}{2D^3(\alpha)\omega_0(\alpha)(v^2_0(\alpha)-D(\alpha))}\\
    &+\frac{3(\omega_0'(\alpha))^2v_0(\alpha)(c^2-v^2)(v^2-D(\alpha))^2}{4D^2(\alpha)\omega_0^2(\alpha)(v^2_0(\alpha)-D(\alpha))}+\frac{3\omega_0'(\alpha)v_0^2(\alpha)v'_0(\alpha)(c^2-v^2)(v^2-D(\alpha))^2}{4D^2(\alpha)\omega_0(\alpha)(v^2_0(\alpha)-D(\alpha))^2}\\
    &-\frac{3\omega_0'(\alpha)D'(\alpha)v_0(\alpha)(c^2-v^2)(v^2-D(\alpha))^2}{4D^2(\alpha)\omega_0(\alpha)(v^2_0(\alpha)-D(\alpha))^2}+\frac{2D'(\alpha)v'_0(\alpha)(c^2-v^2)(D(\alpha)-v^2)(v^2-v_0^2(\alpha))}{(c^2-v_0^2(\alpha))(D(\alpha)-v_0^2(\alpha))^3}\\
    &+\left[-\frac{3\omega_0''(\alpha)(c^2-v^2)(D(\alpha)-v^2)^2}{8D^{\frac{5}{2}}(\alpha)\omega_0(\alpha)}\right.+\frac{3\omega_0'(\alpha)(D(\alpha)-v^2)(D(\alpha)+2c^2-3v^2)vv_\alpha
    }{4D^{\frac{5}{2}}(\alpha)\omega_0(\alpha)}\\
    &-\frac{3\omega_0'(\alpha)D'(\alpha)(c^2-v^2)(D(\alpha)-v^2)}{4D^{\frac{5}{2}}(\alpha)\omega_0(\alpha)}+\frac{15\omega_0'(\alpha)D'(\alpha)(c^2-v^2)(D(\alpha)-v^2)^2}{16D^{\frac{7}{2}}(\alpha)\omega_0(\alpha)}\\
    &\left.+\frac{3(\omega_0'(\alpha))^2(c^2-v^2)(D(\alpha)-v^2)^2}{8D^{\frac{5}{2}}(\alpha)\omega_0^2(\alpha)}\right]\ln\frac{(v-\sqrt{D(\alpha)})(v_0(\alpha)+\sqrt{D(\alpha)})}{(v+\sqrt{D(\alpha)})(v_0(\alpha)-\sqrt{D(\alpha)})}\\
    &+\frac{3\omega_0'(\alpha)(c^2-v^2)(D(\alpha)-v^2)v_\alpha}{4D^2(\alpha)\omega_0(\alpha)}+\frac{3\omega_0'(\alpha)v_0(\alpha)(c^2-v^2)(D(\alpha)-v^2)^2}{4D^2(\alpha)\omega_0(\alpha)(v_0^2-D(\alpha))}\\
    &-\frac{3\omega_0'(\alpha)D'(\alpha)(c^2-v^2)(D(\alpha)-v^2)v}{8D^3(\alpha)\omega_0(\alpha)}-\frac{3\omega_0'(\alpha)D'(\alpha)v_0(\alpha)(c^2-v^2)(D(\alpha)-v^2)^2}{8D^3(\alpha)\omega_0(\alpha)(v_0^2-D(\alpha))}\\
    &+\frac{v''_0(\alpha)(c^2-v^2)(D(\alpha)-v^2)^2}{(c^2-v_0^2(\alpha))(D(\alpha)-v_0^2(\alpha))^2}-\frac{2v'_0(\alpha)(D(\alpha)-v^2)(D(\alpha)+2c^2-3v^2)}{(c^2-v_0^2(\alpha))(D(\alpha)-v_0^2(\alpha))^2}\\
    &+\frac{2v_0(\alpha)(v'_0(\alpha))^2(c^2-v^2)(D(\alpha)-v^2)^2}{(c^2-v_0^2(\alpha))(D(\alpha)-v_0^2(\alpha))^2}+\frac{4v_0(\alpha)(v'_0(\alpha))^2(c^2-v^2)(D(\alpha)-v^2)^2}{(c^2-v_0^2(\alpha))(D(\alpha)-v_0^2(\alpha))^3}.
\end{align*}
\endgroup
\textbf{Case 2}: $D(\alpha)<0$. In this case, we can obtain
\begingroup
\allowdisplaybreaks
\begin{align}\label{raa}
    &r_{\alpha\alpha}(v,\alpha)=-\frac{8m(\alpha-m)v'_0(\alpha)(c^2-v^2)v}{(\alpha v_0^2(\alpha)-2mc^2)^2}-\frac{4m\alpha(\alpha-2m)v''_0(\alpha)(c^2-v^2)v}{(\alpha v_0^2(\alpha)-2mc^2)^2}\nonumber\\
    &-\frac{4m\alpha(\alpha-2m)v'_0(\alpha)(c^2-3v^2)v_\alpha}{(\alpha v_0^2(\alpha)-2mc^2)^2}+\frac{8m\alpha(\alpha-2m)v'_0(\alpha)(v_0^2(\alpha)+2\alpha v_0(\alpha)v'_0(\alpha))(c^2-v^2)v}{(\alpha v_0^2(\alpha)-2mc^2)^3}\nonumber\\
    &+\frac{\left[2\alpha v_0^3(\alpha)+3(\alpha^2-4m^2)v_0^2(\alpha)v'_0(\alpha)-4mc^2v_0(\alpha)-4mc^2(\alpha-2m)v'_0(\alpha)\right](c^2-v^2)v}{(\alpha v_0^2(\alpha)-2mc^2)^2(c^2-v_0^2(\alpha))}\nonumber\\
    &+\frac{\left[(\alpha^2-4m^2)v_0^3(\alpha)-4mc^2(\alpha-2m)v_0(\alpha)\right](c^2-3v^2)v_\alpha}{(\alpha v_0^2(\alpha)-2mc^2)^2(c^2-v_0^2(\alpha))}\nonumber\\
    &-\frac{2\left[(\alpha^2-4m^2)v_0^3(\alpha)-4mc^2(\alpha-2m)v_0(\alpha)\right](c^2-v^2)(v_0^2(\alpha)+2\alpha v_0(\alpha)v'_0(\alpha))v}{(\alpha v_0^2(\alpha)-2mc^2)^3(c^2-v_0^2(\alpha))}\nonumber\\
    &+\frac{2\left[(\alpha^2-4m^2)v_0^3(\alpha)-4mc^2(\alpha-2m)v_0(\alpha)\right](c^2-v^2)v_0(\alpha)v'_0(\alpha)v}{(\alpha v_0^2(\alpha)-2mc^2)^2(c^2-v_0^2(\alpha))}\nonumber\\
    &+\frac{2m(c^2-v^2)\omega_0''(\alpha)}{D^2(\alpha)}-\frac{4m\omega_0'(\alpha)vv_\alpha}{D^2(\alpha)}-\frac{4m(c^2-v^2)\omega_0'(\alpha)D'(\alpha)}{D^3(\alpha)}\nonumber\\
    &-\frac{m\omega_0''(\alpha)(c^2-v^2)v_0(\alpha)v}{D^2(\alpha)(v_0^2(\alpha)-D(\alpha))}-\frac{m\omega_0'(\alpha)(c^2-3v^2)v_0(\alpha)v_\alpha}{D^2(\alpha)(v_0^2(\alpha)-D(\alpha))}-\frac{m\omega_0'(\alpha)(c^2-v^2)v'_0(\alpha)v}{D^2(\alpha)(v_0^2(\alpha)-D(\alpha))}\nonumber\\
    &+\frac{2m\omega_0'(\alpha)(c^2-v^2)v_0(\alpha)D'(\alpha)v}{D^3(\alpha)(v_0^2(\alpha)-D(\alpha))}+\frac{m\omega_0'(\alpha)(c^2-v^2)v_0(\alpha)(2v_0(\alpha)v'_0(\alpha)-D'(\alpha))v}{D^2(\alpha)(v_0^2(\alpha)-D(\alpha))^2}\nonumber\\
    &+\frac{m\omega_0''(\alpha)(c^2-v^2)v^2}{D^2(\alpha)(v_0^2(\alpha)-D(\alpha))}+\frac{2m\omega_0'(\alpha)(c^2-2v^2)vv_\alpha}{D^2(\alpha)(v_0^2(\alpha)-D(\alpha))}-\frac{2m\omega_0'(\alpha)(c^2-v^2)D'(\alpha)v^2}{D^3(\alpha)(v_0^2(\alpha)-D(\alpha))}\nonumber\\
    &-\frac{m\omega_0'(\alpha)(c^2-v^2)(2vv_\alpha-D'(\alpha))v^2}{D^2(\alpha)(v_0^2(\alpha)-D(\alpha))^2}-\frac{3m\omega_0''(\alpha)(c^2-v^2)v}{(-D(\alpha))^{\frac{5}{2}}}\arctan\frac{v_0(\alpha)}{(-D(\alpha))^{\frac{1}{2}}}\nonumber\\
    &-\frac{3m\omega_0'(\alpha)(c^2-3v^2)v_\alpha}{(-D(\alpha))^{\frac{5}{2}}}\arctan\frac{v_0(\alpha)}{(-D(\alpha))^{\frac{1}{2}}}-\frac{15m\omega_0'(\alpha)(c^2-v^2)D'(\alpha)v}{2(-D(\alpha))^{\frac{7}{2}}}\arctan\frac{v_0(\alpha)}{(-D(\alpha))^{\frac{1}{2}}}\nonumber\\
    &-\frac{3m\omega_0'(\alpha)(c^2-v^2)(v_0(\alpha)D'(\alpha)-2D(\alpha)v'_0(\alpha))v}{2(-D(\alpha))^3(v_0^2(\alpha)-D(\alpha))}+\frac{3m\omega_0''(\alpha)(c^2-v^2)v}{(-D(\alpha))^{\frac{5}{2}}}\arctan\frac{v}{(-D(\alpha))^{\frac{1}{2}}}\nonumber\\
    &+\frac{3m\omega_0'(\alpha)(c^2-3v^2)v_\alpha}{(-D(\alpha))^{\frac{5}{2}}}\arctan\frac{v}{(-D(\alpha))^{\frac{1}{2}}}+\frac{15m\omega_0'(\alpha)(c^2-v^2)D'(\alpha)v}{2(-D(\alpha))^{\frac{7}{2}}}\arctan\frac{v}{(-D(\alpha))^{\frac{1}{2}}}\nonumber\\
    &+\frac{3m\omega_0'(\alpha)(c^2-v^2)(vD'(\alpha)-2D(\alpha)v^{\prime}(\alpha))v}{2(-D(\alpha))^3(v_0^2(\alpha)-D(\alpha))},
\end{align}
\endgroup
and

\begingroup
\allowdisplaybreaks
\begin{align}
    v_{\alpha\alpha}=&-\frac{\omega_0''(\alpha)(c^2-v^2)v}{2D(\alpha)\omega_0(\alpha)}-\frac{\omega_0'(\alpha)(c^2-3v^2)v_\alpha}{2D(\alpha)\omega_0(\alpha)}+\frac{\omega_0'(\alpha)D'(\alpha)(c^2-v^2)v}{2D^2(\alpha)\omega_0(\alpha)}+\frac{(\omega_0'(\alpha))^2(c^2-v^2)v}{2D(\alpha)\omega_0^2(\alpha)} \notag \\
    &+\frac{\omega_0''(\alpha)v_0(\alpha)(c^2-v^2)(D(\alpha)-v^2)^2}{2D(\alpha)\omega_0(\alpha)(D(\alpha)-v_0^2(\alpha))^2}+\frac{\omega_0'(\alpha)v'_0(\alpha)(c^2-v^2)(D(\alpha)-v^2)^2}{2D(\alpha)\omega_0(\alpha)(D(\alpha)-v_0^2(\alpha))^2} \notag \\
    &-\frac{\omega_0'(\alpha)v_0(\alpha)(D(\alpha)-v^2)(D(\alpha)+2c^2-3v^2)vv_\alpha}{D(\alpha)\omega_0(\alpha)(D(\alpha)-v_0^2(\alpha))^2}-\frac{\omega_0'(\alpha)v_0(\alpha)(c^2-v^2)(D(\alpha)-v^2)^2D'(\alpha)}{2D^2(\alpha)\omega_0(\alpha)(D(\alpha)-v_0^2(\alpha))^2} \notag \\
    &-\frac{(\omega_0'(\alpha))^2v_0(\alpha)(c^2-v^2)(D(\alpha)-v^2)^2}{2D(\alpha)\omega_0^2(\alpha)(D(\alpha)-v_0^2(\alpha))^2}-\frac{\omega_0'(\alpha)v_0(\alpha)(c^2-v^2)(D(\alpha)-v^2)^2D'(\alpha)}{D(\alpha)\omega_0(\alpha)(D(\alpha)-v_0^2(\alpha))^3} \notag \\
    &+\frac{2\omega_0'(\alpha)v_0^2(\alpha)v'_0(\alpha)(c^2-v^2)(D(\alpha)-v^2)^2}{D(\alpha)\omega_0(\alpha)(D(\alpha)-v_0^2(\alpha))^3}+\frac{3\omega_0''(\alpha)(c^2-v^2)(v^2-D(\alpha))v}{4D^2(\alpha)\omega_0(\alpha)} \notag \\
    &+\frac{3\omega_0'(\alpha)(c^2-3v^2)(v^2-D(\alpha))v_\alpha}{4D^2(\alpha)\omega_0(\alpha)}+\frac{3\omega_0'(\alpha)(c^2-v^2)v^2v_\alpha}{2D^2(\alpha)\omega_0(\alpha)}-\frac{3\omega_0'(\alpha)D'(\alpha)(c^2-v^2)v}{4D^2(\alpha)\omega_0(\alpha)} \notag \\
    &-\frac{3\omega_0'(\alpha)D'(\alpha)(c^2-v^2)(v^2-D(\alpha))v}{2D^3(\alpha)\omega_0(\alpha)}-\frac{3(\omega_0'(\alpha))^2(c^2-v^2)(v^2-D(\alpha))v}{4D^2(\alpha)\omega_0^2(\alpha)} \notag \\
    &-\frac{3\omega_0''(\alpha)v_0(\alpha)(c^2-v^2)(v^2-D(\alpha))^2}{4D^2(\alpha)\omega_0(\alpha)(v^2_0(\alpha)-D(\alpha))}-\frac{3\omega_0'(\alpha)v'_0(\alpha)(c^2-v^2)(v^2-D(\alpha))^2}{4D^2(\alpha)\omega_0(\alpha)(v^2_0(\alpha)-D(\alpha))} \notag \\
    &+\frac{3\omega_0'(\alpha)v_0(\alpha)(v^2-D(\alpha))^2vv_\alpha}{2D^2(\alpha)\omega_0(\alpha)(v^2_0(\alpha)-D(\alpha))}-\frac{3\omega_0'(\alpha)v_0(\alpha)(c^2-v^2)(v^2-D(\alpha))vv_\alpha}{D^2(\alpha)\omega_0(\alpha)(v^2_0(\alpha)-D(\alpha))} \notag \\
    &+\frac{3\omega_0'(\alpha)D'_0(\alpha)v_0(\alpha)(c^2-v^2)(v^2-D(\alpha))^2}{2D^2(\alpha)\omega_0(\alpha)(v^2_0(\alpha)-D(\alpha))}+\frac{3\omega_0'(\alpha)D'(\alpha)v_0(\alpha)(c^2-v^2)(v^2-D(\alpha))^2}{2D^3(\alpha)\omega_0(\alpha)(v^2_0(\alpha)-D(\alpha))} \notag \\
    &+\frac{3(\omega_0'(\alpha))^2v_0(\alpha)(c^2-v^2)(v^2-D(\alpha))^2}{4D^2(\alpha)\omega_0^2(\alpha)(v^2_0(\alpha)-D(\alpha))}+\frac{3\omega_0'(\alpha)v_0^2(\alpha)v'_0(\alpha)(c^2-v^2)(v^2-D(\alpha))^2}{4D^2(\alpha)\omega_0(\alpha)(v^2_0(\alpha)-D(\alpha))^2} \notag \\
    &-\frac{3\omega_0'(\alpha)D'(\alpha)v_0(\alpha)(c^2-v^2)(v^2-D(\alpha))^2}{4D^2(\alpha)\omega_0(\alpha)(v^2_0(\alpha)-D(\alpha))^2} \notag \\
    &+\left[-\frac{3\omega_0''(\alpha)(c^2-v^2)(D(\alpha)-v^2)^2}{8(-D(\alpha))^{\frac{5}{2}}\omega_0(\alpha)}\right.+\frac{3\omega_0'(\alpha)(D(\alpha)-v^2)(D(\alpha)+2c^2-3v^2)vv_\alpha}{4(-D(\alpha))^{\frac{5}{2}}\omega_0(\alpha)} \notag \\
    &-\frac{3\omega_0'(\alpha)D'(\alpha)(c^2-v^2)(D(\alpha)-v^2)}{4(-D(\alpha))^{\frac{5}{2}}\omega_0(\alpha)}+\frac{15\omega_0'(\alpha)D'(\alpha)(c^2-v^2)(D(\alpha)-v^2)^2}{16(-D(\alpha))^{\frac{7}{2}}\omega_0(\alpha)} \notag \\
    &\left.+\frac{3(\omega_0'(\alpha))^2(c^2-v^2)(D(\alpha)-v^2)^2}{8(-D(\alpha))^{\frac{5}{2}}\omega_0^2(\alpha)}\right]\left(\arctan\frac{v}{\sqrt{-D(\alpha)}}-\arctan\frac{v_0(\alpha)}{\sqrt{-D(\alpha)}}\right) \notag \\
    &+\frac{3\omega_0'(\alpha)(c^2-v^2)(D(\alpha)-v^2)v_\alpha}{8D^2(\alpha)\omega_0(\alpha)}+\frac{3\omega_0'(\alpha)v_0(\alpha)(c^2-v^2)(D(\alpha)-v^2)^2}{8D^2(\alpha)\omega_0(\alpha)(v_0^2-D(\alpha))} \notag \\
    &-\frac{3\omega_0'(\alpha)D'(\alpha)(c^2-v^2)(D(\alpha)-v^2)v}{16D^3(\alpha)\omega_0(\alpha)}-\frac{3\omega_0'(\alpha)D'(\alpha)v_0(\alpha)(c^2-v^2)(D(\alpha)-v^2)^2}{16D^3(\alpha)\omega_0(\alpha)(v_0^2-D(\alpha))} \notag \\
    &+\frac{v''_0(\alpha)(c^2-v^2)(D(\alpha)-v^2)^2}{(c^2-v_0^2(\alpha))(D(\alpha)-v_0^2(\alpha))^2}-\frac{2v'_0(\alpha)(D(\alpha)-v^2)(D(\alpha)+2c^2-3v^2)}{(c^2-v_0^2(\alpha))(D(\alpha)-v_0^2(\alpha))^2} \notag \\
    &+\frac{2v_0(\alpha)(v'_0(\alpha))^2(c^2-v^2)(D(\alpha)-v^2)^2}{(c^2-v_0^2(\alpha))(D(\alpha)-v_0^2(\alpha))^2}+\frac{4v_0(\alpha)(v'_0(\alpha))^2(c^2-v^2)(D(\alpha)-v^2)^2}{(c^2-v_0^2(\alpha))(D(\alpha)-v_0^2(\alpha))^3} \notag \\
    &+\frac{2D'(\alpha)v'_0(\alpha)(c^2-v^2)(D(\alpha)-v^2)(v^2-v_0^2(\alpha))}{(c^2-v_0^2(\alpha))(D(\alpha)-v_0^2(\alpha))^3} \label{vaa}
\end{align}
\endgroup
\textbf{Case 3}: $D(\alpha)=0$. In this case, we have
\begingroup
\allowdisplaybreaks
\begin{align*}
	r_{\alpha\alpha}(v,\alpha)=&\left[2m\omega_0''(\alpha)(c^2-v^2)v+2m\omega_0'(\alpha)(c^2-3v^2)v_\alpha\right]\left(   \frac{\left(\frac{\alpha}{2m}\right)^{\frac{5}{2}}}{5c^5}+\frac{1}{5v^5}\right)+\frac{m\omega_0'(\alpha)\alpha^{\frac{3}{2}}(c^2-v^2)v}{(2m)^{\frac{5}{2}}c^5}\\
	&-\frac{2m\omega_0'(\alpha)(c^2-v^2)v_\alpha}{v^5}+\frac{(c^2-3v^2)v_\alpha}{(c^2-v_0^2(\alpha))v_0(\alpha)}-\frac{(c^2-3v_0^2(\alpha))v'_0(\alpha)(c^2-v^2)v}{(c^2-v_0^2(\alpha))^2v_0^2(\alpha)},
\end{align*}
\endgroup
and
\begingroup
\allowdisplaybreaks
\begin{align*}
	v_{\alpha\alpha}=&\frac{v''_0(\alpha)(c^2-v^2)v^4}{(c^2-v_0^2(\alpha))v_0^4(\alpha)}-\frac{2v'_0(\alpha)v^5v_\alpha}{(c^2-v_0^2(\alpha))v_0^4(\alpha)}+\frac{4v'_0(\alpha)(c^2-v^2)v^3v_\alpha}{(c^2-v_0^2(\alpha))v_0^4(\alpha)}+\frac{2(v'_0(\alpha))^2(c^2-v^2)v^4}{(c^2-v_0^2(\alpha))v_0^3(\alpha)}\\
	&-\frac{4(v'_0(\alpha))^2(c^2-v^2)v^4}{(c^2-v_0^2(\alpha))v_0^5(\alpha)}+\frac{2\omega_0''(\alpha)(c^2-v^2)}{5c^2v}-\frac{2\omega_0''(\alpha)(c^2-v^2)v^4}{5c^2v_0^5(\alpha)}-\frac{4\omega_0'(\alpha)v_\alpha}{5c^2}+\frac{4\omega_0'(\alpha)v^5v_\alpha}{5c^2v_0^5(\alpha)}\\
	&-\frac{2\omega_0'(\alpha)(c^2-v^2)v_\alpha}{5c^2v^2}-\frac{8\omega_0'(\alpha)(c^2-v^2)v^3v_\alpha}{5c^2v_0^5(\alpha)}+\frac{2\omega_0'(\alpha)v'_0(\alpha)(c^2-v^2)v^4}{v_0^6(\alpha)}.
\end{align*}
\endgroup
When $D(\alpha)=\frac{v_0^2(\alpha)-\frac{2m}{\alpha}c^2}{1-\frac{2m}{\alpha}}\neq0$, the denominators in the fomulas of $r_{\alpha\alpha}$ and $v_{\alpha\alpha}$ contains the terms $\alpha v_0^2(\alpha)-2mc^2$,   $c^2-v_0^2(\alpha)$,   $D(\alpha)$,   $v_0^2(\alpha)-D(\alpha)$,   $v_0(\alpha)+\sqrt{D(\alpha)}$,   $v-\sqrt{D(\alpha)}$ and $\omega_0(\alpha)$. The nonvanishing property of these terms follows from the initial value assumptions, \eqref{omega2} and Proposition \ref{Prop1}. In the case $D(\alpha)=0$, the potential denominators reduce to $v$,  $v_0(\alpha)$ and $c^2-v_0^2(\alpha)$, whose non-vanishing can been established similarly.
In summary, we can prove that $r_{\alpha\alpha}$ and  $v_{\alpha\alpha}$ remain bounded within any finite time.
\end{proof}
\section{Proof of Theorems \ref{thm:1.1} and \ref{thm:1.2} }
With the above preparations, we are ready to prove the main results.
\subsection{Proof of Theorem \ref{thm:1.1}}
We first prove Theorem \ref{thm:1.1}, which deals with the blow-up phenomenon of the fluid when $v_0(\alpha)=0$.
\subsubsection{Blow up behavior of the classical solution}
Given $v_0(\alpha)=0$, we have established in \textbf{Case 4} of Section \ref{subsection3.2.2} that there exists a finite time $t'$ such that $r_\alpha(t')=0$ under the following two conditions
$$2m<\alpha\leq\frac{\sqrt{33}+3}{3}m\mbox{ and }v'_0(\alpha)<-\frac{mc(3\alpha-2m)}{2\alpha^2(\alpha-2m)}-\frac{3\sqrt{2}c}{2\alpha}\sqrt{\frac{m}{\alpha-2m}}\arctan\sqrt{\frac{\alpha-2m}{2m}}$$
or
$$\alpha>\frac{\sqrt{33}+3}{3}m$$ and $$v'_0(\alpha)\leq-\frac{c}{2\alpha}\sqrt{\frac{m}{\alpha-2m}}\left(\sqrt{\frac{\sqrt{33}-3}{2}}+\frac{(\sqrt{33}-3)^{\frac{3}{2}}}{8\sqrt{3}}+3\sqrt{2}\arctan\sqrt{\frac{\sqrt{33}-3}{6}}\right).$$

Taking $0<\delta_2<\delta_1$, where $\delta_1$ is defined in Section \ref{subsection3.2.2}, we have
\begin{equation*}
	\begin{aligned}
		\rho(t'-\delta_2)&=\rho_0(\alpha)e^{-\int_{0}^{t'-\delta_2} \frac{1-\frac{2m}{r}}{c}\left(\frac{m\omega_0(\alpha)}{r^2v}+\frac{\left(1-\frac{2m}{r}\right)\omega_0'(\alpha)}{2vr_\alpha}+\frac{2v}{r}\right)   ds           }\\
		&=d_1e^{\int_{t'-\delta_1}^{t'-\delta_2}\left(\frac{\left(1-\frac{2m}{r}\right)^2v(t')}{\left(1-\frac{2m}{r(t')}\right)^2v}\frac{1}{t'-s}\right) ds }.
	\end{aligned}
\end{equation*}
Here $d_1$ is a positive constant for fixed $\delta_1$.

According to Corollary \ref{cor:2.1}, we have  $$\rho=O\left(\frac{1}{t'-t}\right) \rightarrow +\infty,$$
when $t'-t=\delta_2\rightarrow  0.$

At the same time, we have
\begin{equation*}
	v_r=\frac{m\omega_0(\alpha)}{r^2v}+\frac{\left(1-\frac{2m}{r}\right)\omega_0'(\alpha)}{2vr_\alpha}=O\left(\frac{1}{t-t'}\right)\rightarrow -\infty.
\end{equation*}

\subsubsection{Global existence of the classical solution}

If the initial data satisfy
$$2m<\alpha\leq\frac{\sqrt{33}+3}{3}m\mbox{ and }v'_0(\alpha)\geq-\frac{mc(3\alpha-2m)}{2\alpha^2(\alpha-2m)}-\frac{3\sqrt{2}c}{2\alpha}\sqrt{\frac{m}{\alpha-2m}}\arctan\sqrt{\frac{\alpha-2m}{2m}}$$ or
$$\alpha>\frac{\sqrt{33}+3}{3}m$$ and $$v'_0(\alpha)>-\frac{c}{2\alpha}\sqrt{\frac{m}{\alpha-2m}}\left(\sqrt{\frac{\sqrt{33}-3}{2}}+\frac{(\sqrt{33}-3)^{\frac{3}{2}}}{8\sqrt{3}}+3\sqrt{2}\arctan\sqrt{\frac{\sqrt{33}-3}{6}}\right).$$
We have $r_\alpha>0$ and $v\neq0$ hold for any finite time $t>0$.

Combining this with \eqref{raa} and \eqref{vaa}, we can conclude that $v_r,\ \rho,\ v_{rr}$ and $\rho_r$ are continuous and bounded on $t \in [0, \infty)$, which means that the classical solution exists globally.


\subsection{Proof of Theorem \ref{thm:1.2}}
Next, we prove Theorem \ref{thm:1.2} by dividing into two cases according to the symbol of $v_0(\alpha)$.

\textbf{Case I:} $v_0(\alpha)<0$

In this case, we have $$-c<v<v_0(\alpha)<0,$$ thus, $v$ never vanishes in finite time. 

\textbf{Case I.1:} $\omega_0'(\alpha)>0\ \&\ D(\alpha)>0$


In \textbf{Case 1} of Section \ref{subsection3.2.2}, we have established that there exists a finite time $t'$ such that $r_\alpha(t')=0$ under condition
\begin{equation*}
	\omega_0'(\alpha)>\frac{1}{Dis(\alpha)},
\end{equation*}
where
\begin{align*}
	Dis(\alpha)=&\frac{2m(c^2-v_0^2(\alpha))}{D^2(\alpha)}+\frac{2m(c^2-v_0^2(\alpha))v_0(\alpha)}{D^2(\alpha)c}+\frac{m(\alpha-2m)v_0^2(\alpha)}{2mD^2(\alpha)}+\frac{mc(\alpha-2m)v_0(\alpha)}{\alpha D^2(\alpha)}\\
	&+\frac{3m(c^2-v_0^2(\alpha))v_0(\alpha)}{2D^{\frac{5}{2}}(\alpha)}\ln \frac{\left(v_0(\alpha)-\sqrt{D(\alpha)}\right)\left(c-\sqrt{D(\alpha)}\right)}{\left(v_0(\alpha)+\sqrt{D(\alpha)}\right)\left(c+\sqrt{D(\alpha)}\right)}.
\end{align*}

Taking $0<\delta_2<\delta_1$, where $\delta_1$ is defined in Section \ref{subsection3.2.2}, we have
\begin{equation*}
	\begin{aligned}
		\rho(t'-\delta_2)&=\rho_0(\alpha)e^{-\int_{0}^{t'-\delta_2} \frac{1-\frac{2m}{r}}{c}\left(\frac{m\omega_0(\alpha)}{r^2v}+\frac{\left(1-\frac{2m}{r}\right)\omega_0'(\alpha)}{2vr_\alpha}+\frac{2v}{r}\right)   ds           }\\
		&=d_2e^{\int_{t'-\delta_1}^{t'-\delta_2}\left(\frac{\left(1-\frac{2m}{r}\right)^2v(t')}{\left(1-\frac{2m}{r(t')}\right)^2v}\frac{1}{t'-s}\right) ds }.
	\end{aligned}
\end{equation*}
Here $d_2$ is a positive constant for fixed $\delta_1$.

According to Corollary \ref{cor:2.1}, we have  $$\rho=O\left(\frac{1}{t'-t}\right) \rightarrow +\infty,$$
when $t'-t=\delta_2\rightarrow  0.$

At the same time, we have
\begin{equation*}
	v_r=\frac{m\omega_0(\alpha)}{r^2v}+\frac{\left(1-\frac{2m}{r}\right)\omega_0'(\alpha)}{2vr_\alpha}=O\left(\frac{1}{t-t'}\right)\rightarrow -\infty.
\end{equation*}

Otherwise, if
\begin{equation*}
	\omega_0'(\alpha)\leq\frac{1}{Dis(\alpha)},
\end{equation*}
then for any finite time, $r_\alpha>0$, $v\neq0$ are bounded. According to Section \ref{subsec3}, $v_{rr}$ and $\rho_r$ also stay bounded, which means the classical solution exists globally.

\textbf{Case I.2}: $\omega_0'(\alpha)>0\ \&\ D(\alpha)=0$

In \textbf{Case 2} of Section \ref{subsection3.2.2}, we have also established that there exists a finite time $t'$ such that $r_\alpha(t')=0$ under condition
\begin{equation*}
	\omega_0'(\alpha)>\frac{1}{Dis(\alpha)}=\frac{10mc^2\alpha^{3/2}}{(\alpha-2m)\left(\alpha^{5/2}-(2m)^{5/2}\right)  }.
\end{equation*}

By the same discussion as \textbf{Case I.1}, we have
$$\rho=O\left(\frac{1}{t'-t}\right) \rightarrow +\infty,\ v_r=\frac{m\omega_0(\alpha)}{r^2v}+\frac{\left(1-\frac{2m}{r}\right)\omega_0'(\alpha)}{2vr_\alpha}=O\left(\frac{1}{t-t'}\right)\rightarrow -\infty, $$
when $t'-t\rightarrow  0.$

If
\begin{equation*}
	\omega_0'(\alpha)\leq\frac{10mc^2\alpha^{3/2}}{(\alpha-2m)\left(\alpha^{5/2}-(2m)^{5/2}\right)  }.
\end{equation*}
then for any finite time, it holds that $r_\alpha>0$ and $v\neq0$. By the same discussion as \textbf{Case I.1}, we can get a global solution.

\textbf{Case I.3}: $\omega_0'(\alpha)>0\ \&\ D(\alpha)<0$

In \textbf{Case 3} of Section \ref{subsection3.2.2}, we know that there exists a finite time $t'$ such that $r_\alpha(t')=0$ under condition
\begin{equation*}
	\omega_0'(\alpha)>\frac{1}{Dis(\alpha)},
\end{equation*}
where
\begin{align*}
	Dis(\alpha)=&\frac{2m(c^2-v_0^2(\alpha))}{D^2(\alpha)}+\frac{2m(c^2-v_0^2(\alpha))v_0(\alpha)}{D^2(\alpha)c}+\frac{m(\alpha-2m)v_0^2(\alpha)}{2mD^2(\alpha)}+\frac{mc(\alpha-2m)v_0(\alpha)}{\alpha D^2(\alpha)}\\
	&+\frac{3m(c^2-v_0^2(\alpha))v_0(\alpha)}{(-D(\alpha))^{\frac{5}{2}}}\left(\arctan \frac{v_0(\alpha)}{(-D(\alpha))^{\frac{1}{2}}}+\arctan \frac{c}{(-D(\alpha))^{\frac{1}{2}}}\right).
\end{align*}

By the same discussion as \textbf{Case I.1}, we have
$$\rho=O\left(\frac{1}{t'-t}\right) \rightarrow +\infty,$$
and$$ v_r=\frac{m\omega_0(\alpha)}{r^2v}+\frac{\left(1-\frac{2m}{r}\right)\omega_0'(\alpha)}{2vr_\alpha}=O\left(\frac{1}{t-t'}\right)\rightarrow -\infty, $$
when $t'-t\rightarrow  0.$

Otherwise, if
\begin{equation*}
	\omega_0'(\alpha)\leq\frac{1}{Dis(\alpha)},
\end{equation*}
$r_\alpha>0$ and $v\neq0$ hold for any finite time. By the same discussion as \textbf{Case I.1}, we can get a global solution.

\textbf{Case I.4}: $\omega_0'(\alpha)=0$

Given $v_0(\alpha)<0$ and $\omega_0'(\alpha)=0$, \textbf{Case 6} in Section \ref{subsection3.2.2} implies $r_\alpha>0$ at any finite time. Consequently, $v_r$, $\rho$, $v_{rr}$ and $\rho_r$ remains bounded in any finite time and the Cauchy problem admits a global solution.

\textbf{Case I.5}: $\omega_0'(\alpha)<0$

Given $v_0(\alpha)<0$ and $\omega_0'(\alpha)<0$, \textbf{Case 5} in Section \ref{subsection3.2.2} implies $r_\alpha>0$ at any finite time. Consequently, $v_r$, $\rho$, $v_{rr}$ and $\rho_r$ remains bounded in any finite time and the Cauchy problem admits a global solution.

\textbf{Case II:} $v_0(\alpha)>0$

In this case, we still need to split into three subcases depending on the symbol of $\omega_0^{'}(\alpha)$.

\textbf{Case II.1:} $\omega_0'(\alpha)>0$

Given $v_0(\alpha)>0$ and $\omega_0'(\alpha)>0$, \textbf{Case 5} in Section \ref{subsection3.2.2} implies $r_\alpha>0$ for all finite time. Consequently, by \eqref{vr}, $v_r$ stays bounded in any finite time. According to \eqref{X1}, $X$ is bounded for all finite time, which leads to $\rho$ also stays bounded.

\textbf{Case II.1.1}: $\omega_0'(\alpha)>0\ \&\ D(\alpha)\geq0$

Given $v_0(\alpha)>0$ and $D(\alpha)\geq0$, Proposition \ref{Prop4} implies $v>0$ at any finite time. By Section \ref{subsec3}, $v_{rr}$ and $\rho_r$ stay bounded too. This means that we can get a global classical solution along this characteristic.

\textbf{Case II.1.2}: $\omega_0'(\alpha)>0\ \&\ D(\alpha)<0$

Given $v_0(\alpha)>0$ and $D(\alpha)<0$, Proposition \ref{Prop4} implies there exists $t_0>0$ such that $v(t_0,\alpha)=0$. By Section \ref{subsec3}, $v_{rr}$ and $\rho_r$ stay bounded as $v\to 0$. Therefore, the characteristic lines can extend all the way to $v=0$.

When $v(t_0)=0$, from \eqref{r}, \eqref{vr}, \eqref{Im}, \eqref{ra-}, \eqref{raa} and \eqref{vaa} we can calculate the values of $r(t_0,\alpha)$, $v_r(t_0,\alpha)$ and $v_{rr}(t_0,\alpha)$ at this point and take them as the new initial data. We note $\beta=r(t_0,\alpha)$, $v_0(\beta)=v(t_0,\alpha)=0$, then we have
$$v'_0(\beta)=\frac{d v_0(\beta)}{d \beta}=\frac{d v}{d r}(t_0, \alpha)=v_r(t_0, \alpha), $$
$$v''_0(\beta)=\frac{d^2 v_0(\beta)}{d \beta^2}=\frac{d^2 v}{d r^2}(t_0, \alpha)=v_{rr}(t_0, \alpha).$$
Finally, the blow-up condition can be determined according to the criteria for $v_0(\beta)=0$ that will be proven below.

\textbf{Case II.2:} $\omega_0'(\alpha)=0$

In this case, we have
\begin{equation*}
	vv_\alpha=\frac{m\omega_0(\alpha)}{r^2}r_\alpha,\ v_r=\frac{m\omega_0(\alpha)}{r^2v}.
\end{equation*}

\textbf{Case II.2.1}: $\omega_0'(\alpha)=0\ \&\ D(\alpha)\geq0$

Given $v_0(\alpha)>0$, $\omega_0'(\alpha)=0$ and $D(\alpha)\geq0$, \textbf{Case 6} in Section \ref{subsection3.2.2} has shown that $v>0$ and $r_\alpha$ stays bounded for all finite time. Thus, $v_r$ and $\rho$ stay bounded at any finite time. By Section \ref{subsec3}, we can see that $v_{rr},\ \rho_r$ also stay bounded. Thus, we can get a classical solution along this characteristic.

\textbf{Case II.2.2}: $\omega_0'(\alpha)=0\ \&\ D(\alpha)<0$

Given $v_0(\alpha)>0$ and $D(\alpha)<0$, Proposition \ref{Prop4} implies there exists a time $t^*>0$ such that $v(t^*)=0$. When $\omega_0'(\alpha)=0$, the analysis in Section \ref{subsec2} establishes $r_\alpha(t^*)=0$ too. Thus, we have
$$
v(t^*,\alpha)=r_{\alpha}(t^*,\alpha)=0.
$$

Taking $0<\delta_2<\delta_1$, where $\delta_1$ is defined in Section \ref{subsection3.2.2}, we have
\begin{equation*}
	\begin{aligned}
		\rho(t^*-\delta_2)&=\rho_0(\alpha)e^{-\int_{0}^{t^*-\delta_2} \frac{1-\frac{2m}{r}}{c}\left(\frac{m\omega_0(\alpha)}{r^2v}+\frac{2v}{r}\right)   ds          }\\
		&=d_3e^{\int_{t^*-\delta_1}^{t^*-\delta_2}\left(\frac{4m^2\omega^3_0(\alpha)\left(1-\frac{2m}{r}\right)}{c^2(\omega_0(\alpha)+c^2)^2r^2}\frac{1}{s-t^*}\right) ds }.
	\end{aligned}
\end{equation*}
Here $d_3$ is a positive constant for fixed $\delta_1$.

Since $\omega_0(\alpha)<0$ and $r>2m$, according to Corollary \ref{cor:2.1}, we have  $$\rho=O\left(\frac{1}{t^*-t}\right) \rightarrow +\infty$$
when $t^*-t=\delta_2\rightarrow  0.$

At the same time, we have
\begin{equation*}
	v_r=\frac{m\omega_0(\alpha)}{r^2v}=O\left(\frac{1}{t-t^*}\right)\rightarrow -\infty.
\end{equation*}

\textbf{Case II.3:} $\omega_0'(\alpha)<0$

\textbf{Case II.3.1}: $\omega_0'(\alpha)<0\ \&\ D(\alpha)\geq0$

Given $v_0(\alpha)>0$, $\omega_0'(\alpha)<0$ and $D(\alpha)\geq0$, \textbf{Case 8} in Section \ref{subsection3.2.2} implies there exists $t'$ such that $r_\alpha(t',\alpha)=0$ while $v(t',\alpha)>0$.

Take $0<\delta_2<\delta_1$, where $\delta_1$ is defined in Section \ref{subsection3.2.2}, we have
\begin{equation*}
	\begin{aligned}
		\rho(t'-\delta_2)&=\rho_0(\alpha)e^{-\int_{0}^{t'-\delta_2} \frac{1-\frac{2m}{r}}{c}\left(\frac{m\omega_0(\alpha)}{r^2v}+\frac{\left(1-\frac{2m}{r}\right)\omega_0'(\alpha)}{2vr_\alpha}+\frac{2v}{r}\right)   ds           }\\
		&=d_4e^{\int_{t'-\delta_1}^{t'-\delta_2}\left(\frac{\left(1-\frac{2m}{r}\right)^2v(t')}{\left(1-\frac{2m}{r(t')}\right)^2v}\frac{1}{t'-s}\right) ds }.
	\end{aligned}
\end{equation*}
Here $d_4$ is a positive constant for fixed $\delta_1$.

According to Corollary \ref{cor:2.1}, we have  $$\rho=O\left(\frac{1}{t'-t}\right) \rightarrow +\infty,$$
when $t'-t=\delta_2\rightarrow  0.$

At the same time, we have
\begin{equation*}
	v_r=\frac{m\omega_0(\alpha)}{r^2v}+\frac{\left(1-\frac{2m}{r}\right)\omega_0'(\alpha)}{2vr_\alpha}=O\left(\frac{1}{t-t'}\right)\rightarrow -\infty.
\end{equation*}

\textbf{Case II.3.2}: $\omega_0'(\alpha)<0\ \&\ D(\alpha)<0$

Given $v_0(\alpha)>0$ and $D(\alpha)<0$, \textbf{Case 8} in Section \ref{subsection3.2.2} implies there exists  $t'$ such that $r_\alpha(t',\alpha)=0$ while $v(t',\alpha)>0$. By the same discussion as \textbf{Case II.3.1}, we have
$$ v_r=\frac{m\omega_0(\alpha)}{r^2v}+\frac{\left(1-\frac{2m}{r}\right)\omega_0'(\alpha)}{2vr_\alpha}=O\left(\frac{1}{t-t'}\right)\rightarrow -\infty,\ \rho=O\left(\frac{1}{t'-t}\right) \rightarrow +\infty $$
when $t'-t\rightarrow  0.$
\begin{remark}
    In above analysis, the blowup time $t'$ may vary from different cases.
\end{remark}

\vspace{5mm}
\phantomsection
\noindent{\large {\bf Acknowledgements.}} S. Miao is supported by the National Key R \& D Program of China 2021YFA1001700 and NSF of China under Grants 12426203, 12221001. C. Wei is supported by Zhejiang Provincial Natural Science Foundation
of China under Grant No. ZCLZ26A0101 and Department of Education of Zhejiang Province under Grant No. JGCG2025595.
%
%
%
%
%

\end{document}